\documentclass[a4paper, 12pt]{amsart}
\usepackage{mathrsfs}
\usepackage{amsmath}
\usepackage{amssymb}
\usepackage{verbatim}
\usepackage{xcolor}
\usepackage{CJK}
\usepackage{graphicx}
\usepackage[all]{xy}
\usepackage{appendix}
\DeclareGraphicsExtensions{.eps,.ps,.jpg,.bmp}

\theoremstyle{plain}

\newtheorem{theorem}{Theorem}[section]
\newtheorem{proposition}{Proposition}[section]

\theoremstyle{definition}
\newtheorem{definition}{Definition}[section]

\theoremstyle{remark}

\newtheorem*{acknowledgment}{Acknowledgment}

\title{The harmonic theory on a vector bundle with singular Hermitian metrics and positivity}
\author{Jingcao Wu}

\begin{document}
\pagestyle{plain}
\begin{abstract}
Let $E$ be a holomorphic vector bundle endowed with a singular Hermitian metric $H$. In this paper, we develop the harmonic theory on $(E,H)$. Then we extend several canonical results of J. Koll\'{a}r and K. Takegoshi to this situation. In the end, we generalise Nakano's vanishing theorem.
\end{abstract}
\maketitle

\section{Introduction}
This is a continuation of our work \cite{Wu20b} about the singular Hermitian metric on a holomorphic vector bundle. 

Let $E$ be a holomorphic vector bundle of rank $r+1$ over a compact K\"{a}hler manifold $(Y,\omega)$ of dimension $n$. Let $X:=\mathbb{P}(E^{\ast})$ be the projectivsed bundle with the natural projection $\pi:X\rightarrow Y$ and the tautological line bundle $\mathcal{O}_{E}(1):=\mathcal{O}_{X}(1)$. Let $\Omega$ be a K\"{a}hler metric on $X$. We proposed an alternative definition for the singular Hermitian metric $H$ on $E$ in \cite{Wu20b}, and showed that this type of singular metric differs slightly with the one defined in \cite{Rau15}. Moreover, it has good nature to help us to define the Griffiths and Nakano positivities. The goal of this paper is to develop the harmonic theory on $(E,H)$.

We first briefly recall the canonical harmonic theory when $H$ is smooth. In this context, the adjoint operator $\bar{\partial}^{\ast}$ of the $\bar{\partial}$ operator with respect to the $L^{2}$-norm $\|\cdot\|_{H,\omega}$ is defined as
\[
\bar{\partial}^{\ast}:=\ast\partial_{H}\ast,
\]
where $\ast$ is the Hodge $\ast$-operator defined by $\omega$ and $\partial_{H}$ is the $(1,0)$-part of the Chern connection associated with $H$. Then the $\bar{\partial}$-Laplacian operator \cite{GrH78}, defined as 
\[
\Box=\bar{\partial}\bar{\partial}^{\ast}+\bar{\partial}^{\ast}\bar{\partial}:A^{p,q}(Y,E)\rightarrow A^{p,q}(Y,E),
\]
is a self-adjoint elliptic operator. Here $A^{p,q}(Y,E)$ is the collection of all the smooth $E$-valued $(p,q)$-forms on $Y$. Thus, the eigenform of $\Box$ with eigenvalue zero is called a harmonic form, and the harmonic space is defined as
\[
\mathcal{H}^{p,q}(Y,E):=\{\alpha\in A^{p,q}(Y,E);\Box\alpha=0\}.
\]
The celebrated Hodge's theorem states that
\[
\mathcal{H}^{p,q}(Y,E)\simeq H^{p,q}(Y,E),
\]
where $H^{p,q}(Y,E)$ is the Dolbeault cohomology group \cite{GrH78}.

Now $H$ is not necessary to be smooth. The first step is to approximate it with a family of smooth metrics. When $E$ is a line bundle, it has been done in \cite{Dem92,DPS01}. We generalise their work to the higher rank vector bundle as follows:
\begin{proposition}\label{p11}
Suppose that $E$ is equipped with a singular Hermitian metric $H$ such that $i\Theta_{\mathcal{O}_{E}(1),\varphi}\geqslant v$ for some smooth real $(1,1)$-form $v$ on $X$. Here $\varphi$ is the metric on $\mathcal{O}_{E}(1)$ corresponding to $H$. Moreover, assume that there exits a section $\xi$ of some multiple $\mathcal{O}_{E}(k)$ such that $\sup_{X}|\xi|_{k\varphi}<\infty$. Then, for each positive real number $\varepsilon$ and positive integer $l$, there is a (singular) Hermitian metric $H^{l}_{\varepsilon}$ on $S^{l}E$ such that
\begin{enumerate}
  \item[(a)] $H^{l}_{\varepsilon}$ is smooth on $Y^{\prime}$, where $Y^{\prime}$ is an open subvariety of $Y$ independent of $l$ and $\varepsilon$;
  \item[(b)] the sequence of metrics $\{\varphi_{\varepsilon}\}$ on $\mathcal{O}_{E}(1)$ that defines $H^{l}_{\varepsilon}$ converges locally uniformly, decreasingly to $\varphi$ on $\pi^{-1}(Y^{\prime})$. Equivalently, $H^{l}_{\varepsilon}$ converges locally uniformly, increasingly to $S^{l}H$ on $Y^{\prime}$ for every $l$;
  \item[(c)] $\mathscr{I}(\varphi)=\mathscr{I}(\varphi_{\varepsilon})$ for all $\varepsilon$;
  \item[(d)] for every $\varepsilon>0$, 
  \[
  i\Theta_{\mathcal{O}_{E}(1),\varphi_{\varepsilon}}\geqslant v-\varepsilon\Omega.
  \]
\end{enumerate}
If $(E,H)$ is strongly positive in the sense of Nakano (see Definition \ref{d24}), we moreover have
\begin{enumerate}
  \item[(e)] for every relatively compact subset $Y^{\prime\prime}\subset\subset Y^{\prime}$ and every $l$,
  \[
  i\Theta_{S^{l}E,H^{l}_{\varepsilon}}\geqslant-C_{l}\varepsilon\omega\otimes\mathrm{Id}_{S^{l}E}
  \]
  over $Y^{\prime\prime}$ in the sense of Nakano (see Definition \ref{d21}) for a constant $C_{l}$.
\end{enumerate}
\end{proposition}
	Here $S^{l}H$ is the natural metric on the $l$-th symmetric product $S^{l}E$ induced by $H$. One could find many similarities between the strongly Nakano positivity and the $\min\{n,r+1\}$-nefness defined in \cite{Cat98} based on Proposition \ref{p11}, (e), but we will not give any further discussions in this paper.

Next, for two $E$-valued $(n,q)$-forms $\alpha,\beta$ (not necessary to be $\bar{\partial}$-closed), we say they are cohomologically equivalent if there exits an $E$-valued $(n,q-1)$-form $\gamma$ such that $\alpha=\beta+\bar{\partial}\gamma$. We denote by $\alpha\in[\beta]$ this equivalence relationship. Now $H$ is approximated by $\{H_{\varepsilon}\}$. Since $H_{\varepsilon}$ is smooth on $Y^{\prime}$, the associated Laplacian $\Box_{\varepsilon}$ is well-defined. The harmonic form associated to $H$ is defined as follows:
\begin{definition}\label{d11}
Let $\alpha$ be an $E$-valued $(n,q)$-form on $Y$ such that the $L^{2}$-norm against $H$ is bounded. Assume that for every $\varepsilon$, there exists a cohomological equivalent class $\alpha_{\varepsilon}\in[\alpha|_{Y^{\prime}}]$ such that
\begin{enumerate}
  \item[(a)] $\Box_{\varepsilon}\alpha_{\varepsilon}=0$ on $Y^{\prime}$;
  \item[(b)] $\alpha_{\varepsilon}\rightarrow\alpha|_{Y^{\prime}}$ in the $L^{2}$-norm against $H$.
\end{enumerate}
Then we call $\alpha$ a $\Box_{0}$-harmonic form, and denote it by $\Box_{0}\alpha=0$. Notice that this equality is taken in the sense of $L^{2}$-topology. The space of all the $\Box_{0}$-harmonic forms is denoted by
\[
\mathcal{H}^{n,q}(Y,E(H),\Box_{0}).
\]
\end{definition}
The analytic sheaf $E(H)$ is defined in \cite{Cat98} as
\[
E(H)_{y}:=\{u\in E_{y};\|u\|^{2}_{H,\omega}\textrm{ is integrable in some neighbourhood of }y\}
\]
for a given point $y\in Y$. Although it has already been proved in \cite{Cat98,HoI20} that $E(H)$ is coherent \cite{Har77} in several situations, we would like to present a general version that describes the coherence in Proposition \ref{p23}. In particular, $E(H)$ will always be coherent in our context in the view of Proposition \ref{p23}. Although the $\Box_{0}$-harmonic space is defined only on $Y^{\prime}$, we can still prove the following proposition when $(E,H)$ is strongly Nakano positive (see Definition \ref{d24}).
\begin{proposition}[A singular version of Hodge's theorem]\label{p12}
Assume that $(E,H)$ is strongly positive in the sense of Nakano. The following isomorphism holds:
\begin{equation}\label{e11}
   \mathcal{H}^{n,q}(Y,E(H),\Box_{0})\simeq H^{q}(Y,K_{Y}\otimes E(H)).
\end{equation}
In particular, when $H$ is smooth, we have $E(H)=E$. Thus, 
\[
\alpha\in\mathcal{H}^{n,q}(Y,E,\Box_{0})
\] 
if and only if $\alpha$ is harmonic in the usual sense.
\end{proposition}

The group $H^{q}(Y,K_{Y}\otimes E(H))$ is interpreted as a cohomology group associated with a coherent sheaf $K_{Y}\otimes E(H)$ as is explained in \cite{Har77}. Moreover we obtain the following regularity property due to the canonical Bochner technique.
\begin{proposition}\label{p13}
Assume that $(E,H)$ is a (singular) Hermitian vector bundle that is strongly positive in the sense of Nakano. Let $\alpha$ be an $E$-valued $(n,q)$-form whose $L^{2}$-norm against $H$ is bounded. Then
\begin{enumerate}
  \item if $\alpha$ is $\Box_{0}$-harmonic, $\bar{\partial}(\ast\alpha)=0$. In particular, $\ast\alpha$ is holomorphic.
  \item if $\alpha$ is a weak solution of $\Box_{0}\alpha=0$, $\alpha$ must be smooth.
\end{enumerate}
\end{proposition}
Here $\ast$ refers to the Hodge $\ast$-operator defined by $\omega$. Based on the harmonic theory constructed above, we then generalise several canonical results of Koll\'{a}r \cite{Ko86a,Ko86b} and Takegoshi \cite{Tak95} as follows:
\begin{theorem}\label{t11}
Let $f:Y\rightarrow Z$ be a fibration between two compact K\"{a}hler manifolds. Let $n=\dim Y$ and $m=\dim Z$. Suppose that $(E,H)$ is a (singular) Hermitian vector bundle over $Y$ that is strongly positive in the sense of Nakano. Moreover, assume that $H|_{Y_{z}}$ is well-defined for every $z\in Z$. Then the following theorems hold:

{\rm \textbf{I Decomposition Theorem}}. The Leray spectral sequence \cite{GrH78}
  \[
  E^{p,q}_{2}=H^{p}(Z,R^{q}f_{\ast}(K_{Y}\otimes E(H)))\Rightarrow H^{p+q}(Y,K_{Y}\otimes E(H))
  \]
  degenerates at $E_{2}$. As a consequence, it holds that
  \[
  \dim H^{i}(Y,K_{Y}\otimes E(H))=\sum_{p+q=i}\dim H^{p}(Z,R^{q}f_{\ast}(K_{Y}\otimes E(H)))
  \]
  for any $i\geqslant0$.
  
{\rm \textbf{II Torsion freeness Theorem}}. For $q\geqslant0$ the sheaf homomorphism
  \[
  L^{q}: f_{\ast}(\Omega^{n-q}\otimes E(H))\rightarrow R^{q}f_{\ast}(K_{Y}\otimes E(H))
  \]
  induced by the $q$-times left wedge product by $\omega$ admits a splitting sheaf homomorphism
  \[
  S^{q}: R^{q}f_{\ast}(K_{Y}\otimes E(H))\rightarrow f_{\ast}(\Omega^{n-q}\otimes E(H))\textrm{ with }L^{q}\circ S^{q}=\mathrm{id}.
  \]
  In particular, $R^{q}f_{\ast}(K_{Y}\otimes E(H))$ is torsion free \cite{Kob87} for $q\geqslant0$ and vanishes if $q>n-m$. Furthermore, it is even reflexive if $E(H)=E$.
  
{\rm \textbf{III Injectivity Theorem}}. Let $(L,h)$ be a (singular) Hermitian line bundle over $Y$. Recall that $Y^{\prime}$ is the open subvariety appeared in Proposition \ref{p11}. Assume the following conditions:
\begin{enumerate}
  \item[(a)] the singular part of $h$ is contained in $Y-Y^{\prime}$;
  \item[(b)] $i\Theta_{L,h}\geqslant\gamma$ for some real smooth $(1,1)$-form $\gamma$ on $Y$;
  \item[(c)] for some rational $\delta\ll1$, the $\mathbb{Q}$-twisted bundle
  \[
  E<-\delta L>|_{Y_{z}}
  \]
  is strongly positive in the sense of Nakano for every $z$.
\end{enumerate}
For a (non-zero) section $s$ of $L$ with $\sup_{Y}|s|_{h}<\infty$, the multiplication map induced by the tensor product with $s$
  \[
 R^{q}f_{\ast}(s): R^{q}f_{\ast}(K_{Y}\otimes E(H))\rightarrow R^{q}f_{\ast}(K_{Y}\otimes(E\otimes L)(H\otimes h))
  \]
  is well-defined and injective for any $q\geqslant0$.
  
{\rm \textbf{IV Relative vanishing Theorem}}. Let $g:Z\rightarrow W$ be a fibration to a compact K\"{a}hler manifold $W$. Then the Leray spectral sequence:
  \[
  R^{p}g_{\ast}R^{q}f_{\ast}(K_{Y}\otimes E(H))\Rightarrow R^{p+q}(g\circ f)_{\ast}(K_{Y}\otimes E(H))
  \]  
  degenerates. 
\end{theorem}
Theorem II can also been seen as a singular version of the hard Lefschetz theorem \cite{GrH78}. The definition of the $\mathbb{Q}$-twisted bundle in Theorem III can be found in \cite{Laz04}. Except that, $(E\otimes L)(H\otimes h)$ is interpreted as the following sheaf:
\[
\begin{split}
(E\otimes L)(H\otimes h)_{y}:=&\{u\in(E\otimes L)_{y};\|u\|^{2}_{H\otimes h,\omega}\textrm{ is integrable}\\
&\textrm{ in some neighbourhood of }y\}.
\end{split}
\]
It is also coherent by Proposition \ref{p23}.

In the end we discuss various vanishing theorems. 
\begin{theorem}\label{t12}
Let $f:Y\rightarrow Z$ be a fibration between two compact K\"{a}hler manifolds.  

{\rm \textbf{I Nadel-type vanishing Theorem}}. Let $(L,h)$ be an $f$-big line bundle, and let $(E,H)$ be a vector bundle that is strongly positive in the sense of Nakano. Assume that $H|_{Y_{z}}$ is well-defined for every $z$. Then
\[
R^{q}f_{\ast}(K_{Y}\otimes(E\otimes L)(H\otimes h))=0\textrm{ for every }q>0.
\]

{\rm \textbf{II Nakano-type vanishing Theorem}}. Assume that $(E,H)$ is strongly strictly positive in the sense of Nakano. Then
\[
H^{q}(Y,K_{Y}\otimes S^{l}E(S^{l}H))=0\textrm{ for every }l,q>0.
\] 

{\rm \textbf{III Griffiths-type vanishing Theorem}}. Assume that $(E,H)$ is strictly positive in the sense of Griffiths (see Definition \ref{d24}). Then
\[
H^{q}(Y,K_{Y}\otimes S^{l}(E\otimes\det E)(S^{l}(H\otimes\det H)))=0\textrm{ for every }l,q>0.
\] 
\end{theorem}
The definition for an $f$-big line bundle can be found in \cite{Fuj13}. Notice that when $E$ is a line bundle, the strong Nakano strict positivity is equivalent to the bigness. At this time Theorem II is just Nadel's vanishing theorem \cite{Nad90}. Furthermore, $(E,H)$ is strongly positive in the sense of Nakano if and only if $(E,H)$ is positive in the sense of Griffiths, if and only if $(E,H)$ is pseudo-effective. In this situation, the topics related to the content of Theorem \ref{t11} and \ref{t12} are fully studied in recent years. See \cite{Eno93,Fuj12,Fuj13,FuM16,GoM17,Mat14,Mat15,Mat16} and the references therein for more details. Our work benefits a lot from them.

\begin{acknowledgment}
The author wants to thank Prof. Jixiang Fu, who brought this problem to his attention and for numerous discussions directly related to this work.
\end{acknowledgment}

\section{Preliminary}
\subsection{Set up}
In the rest of this paper, we will use the following set up.

(I) Let $(Y,\omega)$ and $Z$ be compact K\"{a}hler manifolds with $\dim Y=n$ and $\dim Z=m$. Let $f:Y\rightarrow Z$ be a fibration, which is a surjective holomorphic map with connected fibres. Let $(E,H)$ be a holomorphic vector bundle over $Y$ of rank $r+1$, endowed with a (singular) Hermitian metric $H$ (see Sect.2.3.). Let $X:=\mathbb{P}(E^{\ast})$ be the projectivsed bundle with the natural projection $\pi:X\rightarrow Y$ and tautological line bundle $\mathcal{O}_{E}(1):=\mathcal{O}_{X}(1)$. In particular, let $\varphi$ be the (singular) metric on $\mathcal{O}_{E}(1)$ corresponding to $H$. Let $\Omega$ be a K\"{a}hler metric on $X$.

(II) For every point $y_{0}\in Y$, we take a local coordinate $y=(y_{1},...,y_{n})$ around $y_{0}$. Fix a holomorphic frame $\{u_{0},...,u_{r}\}$ of $E$, the local coordinates of $E$, $X$ and $\mathcal{O}_{E}(1)$ are $(y,U=(U_{0},...,U_{r}))$, $(y,w=(w_{1},...,w_{r}))$ and $(y,w,\xi)$ respectively. For every point $z_{0}\in Z$, we take a local coordinate $z=(z_{1},...,z_{m})$ around $z_{0}$. Moreover, when $z_{0}$ is a regular point of $f$, the system of local coordinate around $Y_{z_{0}}:=f^{-1}(z_{0})$ can be taken as 
\[
\begin{split}
f:Y&\rightarrow Z\\
(z,(y_{m+1},...,y_{n}))&\mapsto z.
\end{split}
\]

We use the following conventions often. 

(i) Denote by $c_{d}=i^{d^{2}}$ for any non-negative integer $d$. Denote
\[ 
dy=dy_{1}\wedge\cdots\wedge d\bar{y}_{n}\textrm{ and }dV_{Y}=c_{n}dy\wedge d\bar{y}.
\] 
It is similar in the other systems of local coordinate.

(ii) Let $G$ be a smooth Hermitian metric on $E$, we write
\[
G_{i}=\frac{\partial G}{\partial y_{i}}, G_{\bar{j}}=\frac{\partial G}{\partial\bar{y}_{j}}, G_{\alpha}=\frac{\partial G}{\partial U_{\alpha}}, G_{\bar{\beta}}=\frac{\partial G}{\partial\bar{U}_{\beta}}
\]
to denote the derivative with respect to 
\[
y_{i},\bar{y}_{j} (1\leqslant i,j\leqslant n)\textrm{ and } U_{\alpha},\bar{U}_{\beta}  (0\leqslant\alpha,\beta\leqslant r).
\] 
The higher order derivative is similar.
\subsection{Hermitian geometry revisit}
We collect some elementary facts here from \cite{Gri69,Kob75,Kob96}. Let $G$ be a smooth Hermitian metric on $E$. Then it induces a dual metric $G^{\ast}$ on $E^{\ast}$ hence a smooth metric $\psi$ on $\mathcal{O}_{E}(1)$. The curvature associated with $G$ is represented as
\[
\Theta_{E,G}=\sum\Theta_{\alpha\bar{\beta}i\bar{j}}dy^{i}\wedge d\bar{y}^{j}\otimes u_{\alpha}\otimes u^{\ast}_{\beta}
\]
with
\[
\Theta_{\alpha\bar{\beta}i\bar{j}}=-G_{\alpha\bar{\beta}i\bar{j}}+\sum_{\delta,\gamma}G^{\gamma\bar{\delta}}G_{\alpha\bar{\delta}i}G_{\gamma\bar{\beta}\bar{j}}.
\]
Now fix a point $y\in Y$, we can always assume that $\{u_{0},...,u_{r}\}$ is an orthonormal basis with respect to $G$ at $y$. The Griffiths and Nakano positivities \cite{Gri69} is defined as follows:
\begin{definition}\label{d21}
Keep notations before,
\begin{enumerate}
  \item $E$ is called (strictly) positive in the sense of Griffiths at $y$, if for any complex vector $z=(z_{1},...,z_{n})$ and section $u=\sum U_{\alpha}u_{\alpha}$ of $E$,
  \[
  \sum i\Theta_{\alpha\bar{\beta}i\bar{j}}U_{\alpha}\bar{U}_{\beta}z_{i}\bar{z}_{j}
  \]
  is (strictly) positive.
  \item $E$ is called (strictly) positive in the sense of Nakano at $y$, if for any $n$-tuple $(u^{1}=\sum U^{1}_{\alpha}u_{\alpha},...,u^{n}=\sum U^{n}_{\alpha}u_{\alpha})$ of sections of $E$,
  \[
  \sum i\Theta_{\alpha\bar{\beta}i\bar{j}}U^{i}_{\alpha}\bar{U}^{j}_{\beta}
  \]
  is (strictly) positive. Equivalently, $i\Theta_{E,G}$ is a (strictly) positive operator on $TY\otimes E$.
\end{enumerate}
\end{definition} 
Kobayashi proposed an intuitive way in \cite{Kob75,Kob96} to characterise the Griffiths positivity. More precisely, let $\{(\log G)^{\bar{\beta}\alpha}\}$ be the inverse matrix of $\{(\log G)_{\alpha\bar{\beta}}\}$. Then for the holomorphic vector field $\frac{\partial}{\partial y_{i}}$ on $Y$, its horizontal lift to $X$ is defined as
\[
\frac{\delta}{\delta y_{i}}:=\frac{\partial}{\partial y_{i}}-\sum_{\alpha,\beta}(\log G)^{\bar{\beta}\alpha}(\log G)_{\bar{\beta}i}\frac{\partial}{\partial w_{\alpha}}.
\]
The dual basis of $\{\frac{\delta}{\delta y_{i}},\frac{\partial}{\partial w_{\alpha}}\}$ will be
\[
\{dy_{i},\delta w_{\alpha}:=dw_{\alpha}+\sum_{i,\beta}(\log G)_{\bar{\beta}i}(\log G)^{\bar{\beta}\alpha}dy_{i}\}.
\]
Let
\[
\Psi:=\sum iK_{\alpha\bar{\beta}i\bar{j}}\frac{w_{\alpha}\bar{w}_{\beta}}{G}dy_{i}\wedge d\bar{y}_{j}
\]
and
\[
\omega_{FS}:=\sum i\frac{\partial^{2}\log G}{\partial w_{\alpha}\partial\bar{w}_{\beta}}\delta w_{\alpha}\wedge\delta\bar{w}_{\beta},
\]
where
\[
K_{\alpha\bar{\beta}i\bar{j}}:=-G_{\alpha\bar{\beta}i\bar{j}}+\sum_{\delta,\gamma}G^{\gamma\bar{\delta}}G_{\alpha\bar{\delta}i}G_{\gamma\bar{\beta}\bar{j}}.
\]
It is easy to verify that they are globally defined $(1,1)$-forms on $X$. Then the celebrated theorem given by Kobayashi says that
\begin{proposition}[Kobayashi, \cite{Kob75,Kob96}]\label{p21}
\begin{equation}\label{e21}
i\partial\bar{\partial}\psi=-\Psi+\omega_{FS}.
\end{equation}
\end{proposition}
As a consequence, we have
\begin{proposition}[Kobayashi]\label{p22}
$(E,G)$ is positive in the sense of Griffiths if and only if $(\mathcal{O}_{E}(1),\psi)$ is positive.
\end{proposition}

The following definition is useful.
\begin{definition}[\cite{CTi08,Don01,Don05}]\label{d22}
\[
c(\psi)_{i\bar{j}}:=<\frac{\delta}{\delta y_{i}},\frac{\delta}{\delta y_{j}}>_{i\partial\bar{\partial}\psi}
\]
is called the geodesic curvature of $\psi$ in the direction of $i,j$.
\end{definition}
Let $a^{\alpha}_{i}:=-\sum_{\beta}(\log G)^{\bar{\beta}\alpha}(\log G)_{\bar{\beta}i}$, we have
\[
c(\psi)_{i\bar{j}}=(\log G)_{i\bar{j}}-\sum_{\alpha}a^{\alpha}_{i}a_{\alpha\bar{j}}.
\]
Therefore the relationship between $c(\varphi)_{i\bar{j}}$ and $\Psi$ will be
\begin{equation}\label{e22}
-\Psi=\sum ic(\psi)_{i\bar{j}}dz_{i}\wedge d\bar{z}_{j}.
\end{equation}

In particular, $\{c(\psi)_{i\bar{j}}\}$ defines a Hermitian form, which is positive if and only if $\Psi$ is negative.

\subsection{Singular Hermitian metric}
We recall the definition of the singular Hermitian metric on $E$ in \cite{Wu20b}.
\begin{definition}\label{d23}
(1) Consider the $L^{1}$-bounded function $\varphi$ on $X$. We define
\[
Y_{\varphi}:=\{y\in Y;\varphi|_{X_{y}}\textrm{ is well-defined }\}.
\]

(2) Fix a smooth metric $h_{0}$ on $\mathcal{O}_{E}(1)$ as the reference metric, we define
\[
\begin{split}
\mathcal{H}(X):=&\{\varphi\in L^{1}(X);\textrm{on each }X_{y}\textrm{ with }y\in Y_{\varphi}, \varphi|_{X_{y}}\textrm{ is smooth and }\\
&(i\Theta_{\mathcal{O}_{X}(1),h_{0}}+i\partial\bar{\partial}\varphi)|_{X_{y}}\textrm{ is strictly positive }\}.
\end{split}
\]

(3) A singular Hermitian metric on $E$ is a map $H$ with form that
\[
H_{\varphi}(u,u)=\int_{X_{y}}|u|^{2}_{h_{0}}e^{-\varphi}\frac{\omega^{r}_{\varphi,y}}{r!},
\]
where $\varphi\in\mathcal{H}(X)$. Here we use the fact that $\pi_{\ast}\mathcal{O}_{E}(1)=E$.
\end{definition}

We have presented in \cite{Wu20b} that this type of singular metric differs slightly with the one defined in \cite{Rau15} and has good nature. In particular, we have successfully defined the Griffiths and Nakano positivities concerning $(E,H_{\varphi})$. More precisely,
\begin{definition}\label{d24}
Let $H_{\varphi}$ be a (singular) Hermitian metric on $E$, and let $\phi$ be the corresponding metric on $\mathcal{O}_{E}(1)$. (Notice that $\phi$ may not equal to $\varphi$.) Let $(L,h)$ be a (singular) Hermitian line bundle over $Y$. Let $q$ be a rational number. Then
\begin{enumerate}
  \item $(E,H_{\varphi})$ is (strictly) positive in the sense of Griffiths, if $i\partial\bar{\partial}\phi$ is (strictly) positive on $X$.
  \item $(E<qL>,H_{\varphi},h)$ is (strictly) positive in the sense of Griffiths, if $i\partial\bar{\partial}\phi+q\pi^{\ast}c_{1}(L,h)$ is (strictly) positive on $X$. 
  \item $(E,H_{\varphi})$ is (strictly) strongly positive in the sense of Nakano, if the $\mathbb{Q}$-twisted vector bundle
\[
(E<-\frac{1}{r+2}\det E>,H_{\varphi})
\]
is (strictly) positive in the sense of Griffiths.
  \item $(E<qL>,H_{\varphi},h)$ is (strictly) strongly positive in the sense of Nakano, if the $\mathbb{Q}$-twisted vector bundle
\[
(E<\frac{q}{r+2}L-\frac{1}{r+2}\det E>,H_{\varphi},h)
\]
is (strictly) positive in the sense of Griffiths.
\end{enumerate}
\end{definition}
The definition for a $\mathbb{Q}$-twisted bundle can be found in \cite{Laz04}. One refers to \cite{Wu20b} for a detailed discussion for these positivities. 

\subsection{Multiplier ideal sheaf}
Remember that in the line bundle situation, the multiplier ideal sheaf \cite{Nad90} is a powerful tool which establishes the relationship between the algebraic and analytic nature of this line bundle. It would be valuable to extend this notion to a vector bundle. There are several attempts. The analytic sheaf $E(H)$ is defined in \cite{Cat98} is a good candidate. The definition is as follows:
\[
E(H)_{y}:=\{u\in E_{y};\|u\|^{2}_{H,\omega}\textrm{ is integrable in some neighbourhood of }y\}
\]
for a given point $y\in Y$. Moreover, it is proved in \cite{Cat98,HoI20} that $E(H)$ is coherent \cite{Har77} when $(E,H)$ possesses certain positivity. We also prove the coherence in a general setting here.
\begin{proposition}\label{p23}
Let $(E,H)$ be a singular Hermitian vector bundle over $Y$. Let $\varphi$ be the corresponding metric on $\mathcal{O}_{E}(1)$. Assume that $i\partial\bar{\partial}\varphi\geqslant\gamma$ for a real $(1,1)$-form $\gamma$ on $X$. Then $E(H)$ is coherent.
\begin{proof}
The proof follows the same method as \cite{Cat98,Dem12}.

Since coherence is a local property, we can assume without loss of generality that $Y=U$ is a domain in $(\mathbb{C}^{n},z=(z_{1},...,z_{n}))$ and $E=U\times\mathbb{C}^{r+1}$. Let $L^{2}(U,\mathbb{C}^{r+1})_{H}$ be the square integrable $\mathbb{C}^{r+1}$-valued holomorphic functions with respect to $H$ on $U$. It generates a coherent subsheaf 
\[
\mathcal{F}\subset\mathcal{O}_{U}(E)=\underbrace{\mathcal{O}_{U}\oplus\cdots\oplus\mathcal{O}_{U}}_{r+1}
\]
as an $\mathcal{O}_{U}$-module \cite{Har77}. It is clear that $\mathcal{F}\subset E(H)$; in order to prove the equality, we need only check that $\mathcal{F}_{y}+E(H)_{y}\cdot\mathfrak{m}^{s+1}_{U,y}=E(H)_{y}$ for every integer $s$, in view of Nakayama's lemma \cite{AtM69}. Here $\mathfrak{m}_{U,y}$ is the maximal ideal of $\mathcal{O}_{U,y}$. 

Let $f\in E(H)_{y}$ be a germ that is defined in a neighbourhood $V$ of $y$ and let $\rho$ be a cut-off function with support in $V$ such that $\rho=1$ in a smaller neighbourhood of $y$. We solve the equation $\bar{\partial}u=\bar{\partial}(\rho f)$ by means of $L^{2}$-estimates (see the proof of Proposition \ref{p24}) against $e^{-2(n+s)\log|z-z(y)|-\psi}H$, where $\psi$ is a smooth strictly plurisubharmonic function satisfying certain positivity as is shown in the proof of Proposition \ref{p24}. Notice that here we use the curvature condition $i\partial\bar{\partial}\varphi\geqslant\gamma$ to guarantee the existence of such a $\psi$. We then get a solution $u$ such that $\int_{U}\frac{|u|^{2}_{H}}{|z-z(y)|^{2(n+s)}}<\infty$. Thus $F=\rho f-u$ is
holomorphic, $F\in L^{2}(U,\mathbb{C}^{r+1})_{H}$ and 
\[
f_{y}-F_{y}=u_{y}\in E(H)_{y}\cdot\mathfrak{m}^{s+1}_{U,y}.
\]
This proves the equality hence the coherence.
\end{proof}
\end{proposition}

\subsection{de Rham--Weil isomorphism}
The de Rham--Weil isomorphism about the $L^{2}$-cohomology group is essentially used in this paper, so we provide here a detailed proof. This proof includes an $L^{2}$-estimate for a singular Hermitian vector bundle, which seems to be of independent interest. The readers who are familiar with this isomorphism could skip this part. Assume that $E(H)$ is coherent, hence the cohomology group $H^{q}(Y,K_{Y}\otimes E(H))$ is well-defined \cite{GrH78}. Moreover, we have
\begin{proposition}[The de Rham--Weil isomorphism]\label{p24}
Fix a Stein covering $\mathcal{U}:=\{U_{j}\}^{N}_{j=1}$ of $Y$. We have
\begin{equation}\label{e23}
H^{q}(Y,K_{Y}\otimes E(H))\simeq\frac{\mathrm{Ker}(\bar{\partial}:A^{n,q}(Y,E(H))\rightarrow A^{n,q+1}(Y,E(H)))}{\mathrm{Im}(\bar{\partial}:A^{n,q-1}(Y,E(H))\rightarrow A^{n,q}(Y,E(H)))}.
\end{equation}
Here we use $A^{n,q}(U,E(H))$ to refer to all of the $(n,q)$-forms on an open set $U$ whose coefficients are in $E(H)$. In other words, 
\[
\alpha\in A^{n,q}(U,E(H))
\] 
if $\alpha$ is an $E$-valued $(n,q)$-form such that
\[
\int_{U}\|\alpha\|^{2}_{H,\omega}<\infty.
\]
\begin{proof}
The proof follows the same method as \cite{Nad90} except that we will essentially use the Koszul complex without mentioning it. Fix a Stein covering $\mathcal{U}:=\{U_{j}\}^{N}_{j=1}$ of $Y$. Then we have the isomorphism between $H^{q}(Y,K_{Y}\otimes E(H))$ and the \v{C}ech cohomology group $\check{H}^{q}(\mathcal{U},K_{Y}\otimes E(H))$. For simplicity, we put
\[
U_{j_{0}j_{1}...j_{q}}:=U_{j_{0}}\cap\cdots\cap U_{j_{q}}.
\]
The class $[\alpha]$ maps to a $q$-cocycle $\alpha=\{\alpha_{j_{0}...j_{q}}\}$ that satisfies
\[
\alpha_{j_{0}...j_{q}}\in H^{n,0}(U_{j_{0}...j_{q}},E(H))\textrm{ and }\delta\alpha=0,
\]
where $\delta$ is the coboundary operator of the \v{C}ech complex. Then by using a partition $\{\rho_{j}\}$ of unity associated to $\mathcal{U}$, we define $\alpha^{1}:=\{\alpha_{j_{0}...j_{q-1}}\}$ by $\alpha_{j_{0}...j_{q-1}}:=\sum\rho_{j}\alpha_{jj_{0}...j_{q-1}}$. By the construction, we have $\delta\alpha^{1}=\alpha$ and $\delta\bar{\partial}\alpha^{1}=\bar{\partial}\delta\alpha^{1}=\bar{\partial}\alpha=0$. Notice that
\[
\bar{\partial}\alpha^{1}=\{\bar{\partial}\alpha_{j_{0}...j_{q-1}}\}=\{\sum(\bar{\partial}\rho_{j})\cdot\alpha_{jj_{0}...j_{q-1}}\},
\]
we have that $\bar{\partial}\alpha^{1}$ is a $(q-1)$-cocycle with
\[
\bar{\partial}\alpha_{j_{0}...j_{q-1}}\in A^{n,1}(U_{j_{0}...j_{q-1}},E(H)).
\]

From the same argument, we can obtain $\alpha^{2}$ with $\delta\alpha^{2}=\bar{\partial}\alpha^{1}$. Moreover, $\bar{\partial}\alpha^{2}$ is a $(q-2)$-cocycle that satisfies
\[
\bar{\partial}\alpha_{j_{0}...j_{q-2}}\in A^{n,2}(U_{j_{0}...j_{q-2}},E(H))\textrm{ and }\delta\bar{\partial}\alpha^{2}=0.
\]
By repeating this process, we finally obtain $\alpha^{q}=\{\alpha_{j_{0}}\}$. Then $\bar{\partial}\alpha_{j_{0}}$ determines the $(n,q)$-form on $U_{j_{0}}$ with $\bar{\partial}\alpha_{j_{0}}\in A^{n,q}(U_{j_{0}},E(H))$. Moreover, since $\delta\bar{\partial}\alpha_{j_{0}}=0$, they patch together to give a $(0,q)$-form 
\[
\alpha_{q}\in A^{n,q}(X,E(H))
\] 
with $\bar{\partial}\alpha_{q}=0$. From the argument above, we have obtained the (well-defined) map
\[
\check{H}^{q}(\mathcal{U},E(H))\rightarrow\frac{\mathrm{Ker}(\bar {\partial}:A^{n,q}(X,E(H))\rightarrow A^{n,q+1}(X,E(H)))}{\mathrm{Im}(\bar{\partial}:A^{n,q-1}(X,E(H))\rightarrow A^{n,q}(X,E(H)))}.
\]

Now we see that this map is actually an isomorphism by using the $L^{2}$-estimate on each Stein subset $U_{j_{0}...j_{k}}$ for $k=0,...,q$. For every
\[
\beta\in\mathrm{Ker}(\bar {\partial}:A^{n,q}(X,E(H))\rightarrow A^{n,q+1}(X,E(H))),
\]
we define $\beta^{0}:=\{\beta_{j_{0}}\}$ by $\beta_{j_{0}}:=\beta|_{U_{j_{0}}}$. By the $L^{2}$-estimate on $U_{j_{0}}$ against $H$, we obtain $\beta^{1}=\{\beta^{1}_{j_{0}}\}$ such that
\[
\begin{split}
\bar{\partial}\beta^{1}&=\beta^{0},\\
\|\beta^{1}\|^{2}_{H_{U_{j_{0}}}}:&=\sum_{j_{0}}\int_{U_{j_{0}}}\|\beta^{1}_{j_{0}}\|^{2}_{H}\leqslant C_{1}\|\beta^{0}\|^{2}_{H_{U_{j_{0}}}}.
\end{split}
\]
Here $C_{1}$ is an independent constant.

We give a short explanation for this $L^{2}$-estimate. We admit Proposition \ref{p11} for the time being. Then there exits a regularising sequence $\{H_{\varepsilon}\}$ which is smooth on an open subvariety $Y^{\prime}\subset U_{j_{0}}$. Moreover, since $U_{j_{0}}$ is a Stein open subset of $\mathbb{C}^{n}$, we can even make $Y^{\prime}=U_{j_{0}}$. Then, at each point $y\in U_{j_{0}}$, we may then choose a coordinate system which diagonalizes simultaneously the hermitians forms $\omega(y)$ and $i\Theta_{E,H_{\varepsilon}}$, in such a way that
\[
\omega(y)=i\sum dy_{j}\wedge d\bar{y}_{j}\textrm{ and }i\Theta_{E,H_{\varepsilon}}=i\sum\lambda^{\varepsilon}_{j}dy_{j}\wedge d\bar{y}_{j}.
\]
Since $\beta_{0}$ is an $(n,q)$-form, at point $y$ we have 
\[
<[i\Theta_{E,H_{\varepsilon}},\Lambda]\beta_{0},\beta_{0}>_{\omega}\geqslant(\lambda^{\varepsilon}_{1}+\cdots+\lambda^{\varepsilon}_{q})|\beta_{0}|^{2},
\]
where $\Lambda$ is the adjoint operator of $\omega\wedge\cdot$. On the other hand, since $U_{j_{0}}$ is Stein, we can always find a smooth strictly plurisubharmonic function $\psi$, such that $(E,e^{-\psi}H)$ is strictly strongly positive in the sense of Nakano on $U_{j_{0}}$. By Proposition \ref{p11}, (e), we then obtain that
\[
i\Theta_{E,e^{-\psi}H_{\varepsilon}}=i\sum\tau^{\varepsilon}_{j}dz_{j}\wedge d\bar{z}_{j}
\]
such that $\tau^{\varepsilon}_{j}\geqslant C^{\prime}$ for a universal positive constant $C^{\prime}$. Now apply the $L^{2}$-estimate \cite{Dem12} against $e^{-\psi}H_{\varepsilon}$, we obtain $\beta^{1}=\{\beta^{1}_{j_{0}}\}$ such that
\[
\begin{split}
\bar{\partial}\beta^{1}&=\beta^{0},\\
\|\beta^{1}\|^{2}_{H_{\varepsilon,U_{j_{0}}}}:&=\sum_{j_{0}}\int_{U_{j_{0}}}\|\beta^{1}_{j_{0}}\|^{2}_{H_{\varepsilon}}\leqslant C_{1}\|\beta^{0}\|^{2}_{H_{\varepsilon,U_{j_{0}}}}.
\end{split}
\]
Recall that $\beta^{0}=\{\beta|_{U_{j_{0}}}\}$ with $\beta\in A^{n,q}(X,E(H))$, so
\[
\lim_{\varepsilon\rightarrow0}\|\beta^{0}\|^{2}_{H_{\varepsilon,U_{j_{0}}}}<\infty.
\] 
Take the limit of the inequality before with respect to $\varepsilon$, we then obtain the desired estimate. 
 
By the construction, we have $\bar{\partial}\delta\beta^{1}=\delta\bar{\partial}\beta^{1}=\delta\beta^{0}=0$. Moreover,
\[
\delta\beta^{1}=\{\beta^{1}_{j_{0}}|_{U_{j_{0}j_{1}}}-\beta^{1}_{j_{1}}|_{U_{j_{0}j_{1}}}\}\in A^{n,q-1}(U_{j_{0}j_{1}},E(H)).
\]
The explicit meaning of this inclusion should be
\[
\beta^{1}_{j_{0}}|_{U_{j_{0}j_{1}}}-\beta^{1}_{j_{1}}|_{U_{j_{0}j_{1}}}\in A^{n,q-1}(U_{j_{0}j_{1}},E(H)),
\]
but we abuse the notation here and in the rest part. Therefore by the same method, we can obtain $\beta^{2}=\{\beta^{2}_{j_{0}j_{1}}\}$ such that
\[
\begin{split}
\bar{\partial}\beta^{2}&=\delta\beta^{1},\\
\|\beta^{2}\|^{2}_{H_{U_{j_{0}j_{1}}}}:&=\sum_{j_{0},j_{1}}\int_{U_{j_{0}j_{1}}}\|\beta^{1}_{j_{0}j_{1}}\|^{2}_{H}\leqslant C_{2}\|\beta^{1}\|^{2}_{H_{U_{j_{0}}}}.
\end{split}
\]
Similarly, $\bar{\partial}\delta\beta^{2}=\delta\bar{\partial}\beta^{2}=\delta\delta\beta^{1}=0$ and
\[
\delta\beta^{2}\in A^{n,q-2}(U_{j_{0}j_{1}j_{2}},E(H)).
\]
By repeating this process, we finally obtain $\beta^{i}$ such that $\bar{\partial}\beta^{i}=\delta\beta^{i-1}$ and $\delta\beta^{i}\in A^{n,0}(U_{j_{0}...j_{i}},E(H))$. Since $\bar{\partial}\delta\beta^{i}=\delta\bar{\partial}\beta^{i}=\delta\delta\beta^{i-1}=0$, we actually have $\delta\beta^{i}\in H^{n,0}(U_{j_{0}...j_{i}},E(H))$. Hence we have obtained
\[
j:\mathrm{Ker}(\bar{\partial}:A^{n,q}(Y,E(H)))\rightarrow A^{n,q+1}(Y,E(H)))\rightarrow\check{H}^{q}(\mathcal{U},K_{Y}\otimes E(H)).
\]
We claim that $\mathrm{Im}(\bar{\partial}:A^{n,q-1}(Y,E(H))\rightarrow A^{n,q}(Y,E(H)))$ maps to the zero space under $j$, hence it is easy to see that the maps $i,j$ together give an isomorphism
\[
H^{q}(Y,K_{Y}\otimes E(H))\simeq\frac{\mathrm{Ker}(\bar{\partial}:A^{n,q}(Y,E(H))\rightarrow A^{n,q+1}(Y,E(H)))}{\mathrm{Im}(\bar{\partial}:A^{n,q-1}(Y,E(H))\rightarrow A^{n,q}(Y,E(H)))}.
\]

Now we prove the claim by diagram chasing. We only prove the case that $i=1$ and $i=2$, the general follows the same method. For every $\beta\in\mathrm{Im}(\bar{\partial}:A^{n,0}(Y,E(H))\rightarrow A^{n,1}(Y,E(H)))$, we define $\beta^{0}:=\{\beta_{j_{0}}\}$ by $\beta_{j_{0}}:=\beta|_{U_{j_{0}}}$. On the other hand, there exits an $E$-valued $(n,0)$-form $\gamma$ on $Y$ such that $\beta=\bar{\partial}\gamma$ and $\|\gamma\|^{2}_{H_{U_{j_{0}}}}<\infty$ for every $j_{0}$. So we can take $\beta^{1}=\{\beta^{1}_{j_{0}}\}:=\{\gamma|_{U_{j_{0}}}\}$. The morphism $j$ in this situation can be simply written as $j(\beta)=\delta\beta^{1}$.  Since $\gamma$ is globally defined hence $\delta\beta^{1}=0$. We have successfully proved that $j(\beta)=0$ when $i=1$.

Next we consider the $\beta\in\mathrm{Im}(\bar{\partial}:A^{n,1}(Y,E(H))\rightarrow A^{n,2}(Y,E(H)))$, we define $\beta^{0}:=\{\beta_{j_{0}}\}$ by $\beta_{j_{0}}:=\beta|_{U_{j_{0}}}$. On the other hand, there exits an $E$-valued $(n,1)$-form $\gamma$ on $Y$ such that $\beta=\bar{\partial}\gamma$ and $\|\gamma\|^{2}_{H_{U_{j_{1}}}}<\infty$ for every $j_{0}$. So we can take $\beta^{1}=\{\beta^{1}_{j_{0}}\}:=\{\gamma|_{U_{j_{0}}}\}$. Now apply the $L^{2}$-estimate on $U_{j_{0}}$ against $H$, we obtain $\gamma^{1}=\{\gamma^{1}_{j_{0}}\}$ such that
\[
\begin{split}
\bar{\partial}\gamma^{1}&=\beta^{1},\\
\|\gamma^{1}\|^{2}_{H_{U_{j_{0}}}}:&=\sum_{j_{0}}\int_{U_{j_{0}}}\|\gamma^{1}_{j_{0}}\|^{2}_{H}\leqslant C_{1}\|\beta^{1}\|^{2}_{H_{U_{j_{0}}}}.
\end{split}
\]
Let $\beta^{2}=\{\beta^{2}_{j_{0}j_{1}}\}=\delta\gamma^{1}$. Since $\bar{\partial}\beta^{2}=\bar{\partial}\delta\gamma^{1}=\delta\bar{\partial}\gamma^{1}=\delta\beta^{1}$ and $\|\beta^{2}\|^{2}_{H_{U_{j_{0}j_{1}}}}<\infty$, the morphism $j$ is $j(\beta)=\delta\beta^{2}=\delta\delta\gamma^{1}=0$ at this time. We have successfully proved that $j(\beta)=0$ when $i=2$.
\end{proof}
\end{proposition}

\subsection{Bochner technique}
This subsection is devoted to introduce the classic Bochner technique in harmonic theory. It mainly comes from \cite{Dem12,GrH78,Tak95} and the references therein. One should pay attention that the content of this subsection works for a K\"{a}hler manifold $(Y,\omega)$ that is not necessary to be compact. 

Let $(E,G)$ be a holomorphic vector bundle with the smooth Hermitian metric $G$. Then $G$ as well as $\omega$ defines an $L^{2}$-norm $\|\cdot\|_{G,\omega}$ on the space 
\[
A^{\ast}(Y,E)=\oplus A^{p,q}(Y,E).
\] 
Let $\ast$ be the Hodge $\ast$-operator, let $e(\theta)$ be the left wedge product acting by the form $\theta\in A^{p,q}(Y)$ and Let $\partial_{G}$ be the $(1,0)$-part of the Chern connection on $E$ associated with $G$. Then the adjoint operators of $\bar{\partial},\partial_{G},e(\theta)$ with respect to $\|\cdot\|_{G,\omega}$ are denoted by $\bar{\partial}^{\ast},\partial^{\ast}_{G}$ and $e(\theta)^{\ast}$, respectively. In particular, $e(\omega)$ is also denoted by $L$ and $e(\omega)^{\ast}$ is denoted by $\Lambda$. The Laplacian operators are defined as follows:
\[
\begin{split}
\Box_{G}&=\bar{\partial}\bar{\partial}^{\ast}+\bar{\partial}^{\ast}\bar{\partial},\\
\bar{\Box}_{G}&=\partial_{G}\partial^{\ast}_{G}+\partial^{\ast}_{G}\partial_{G}.
\end{split}
\]

Then the Bochner formula is as follows:
\begin{proposition}\label{p25}
For any $\alpha\in A^{p,q}(Y,E)$ and positive smooth function $\eta$ on $Y$ with $\chi:=\log\eta$, we have
\begin{equation}\label{e24}
\begin{split}
&\Box_{G}=\bar{\Box}_{G}+[i\Theta_{E,G},\Lambda],\\
&\Box_{e^{-\chi}G}=\bar{\Box}_{e^{-\chi}G}+[i(\Theta_{E,G}+\partial\bar{\partial}\chi),\Lambda]\textrm{ and }\\
&\|\sqrt{\eta}(\bar{\partial}+e(\bar{\partial}\chi))\alpha\|^{2}_{G,\omega}+\|\sqrt{\eta}\bar{\partial}^{\ast}\alpha\|^{2}_{G,\omega}=\|\sqrt{\eta}(\partial^{\ast}_{G}-e(\partial\chi)^{\ast})\alpha)\|^{2}_{G,\omega}\\
&+\|\sqrt{\eta}\partial_{G}\alpha\|^{2}_{G,\omega}+<\eta i[\Theta_{E,G}+\partial\bar{\partial}\varphi,\Lambda]\alpha,\alpha>_{G,\omega},
\end{split}
\end{equation}
when the integrals above are finite.
\end{proposition}
The following formula is due to the K\"{a}hler property of $\omega$:
\begin{proposition}[Donnelly and Xavier's formula, \cite{DoX84}]\label{p26}
For any 
\[
\alpha\in A^{n,q}(Y,E)
\] 
and smooth function $\chi$ on $Y$, we have
\begin{equation}\label{e25}
\begin{split}
[\bar{\partial},e(\bar{\partial}\chi)^{\ast}]&+e(\partial\chi)\partial^{\ast}_{G}=ie(\partial\bar{\partial}\varphi)\Lambda\textrm{ and }\\
\|e(\partial\chi)^{\ast}\alpha\|^{2}_{G,\omega}&=\|e(\bar{\partial}\chi)\alpha\|^{2}_{G,\omega}+\|e(\bar{\partial}\chi)^{\ast}\alpha\|^{2}_{G,\omega}.
\end{split}
\end{equation}
when the integrals above are finite.
\end{proposition}

\section{The regularising technique}
This section is devoted to prove Proposition \ref{p11}. Recall that in \cite{Dem92,DPS01}, such an approximation was already made for a line bundle as follows:
\begin{proposition}[Theorem 2.2.1, \cite{DPS01}]\label{p31}
Let $(L,\varphi)$ be a (singular) Hermitian line bundle over $Y$ such that $i\Theta_{L,\varphi}\geqslant v$ for a real smooth $(1,1)$-form $v$. There exits a family of singular metrics $\{\varphi_{\varepsilon}\}_{\varepsilon>0}$ with the following properties:
\begin{enumerate}
  \item[(a)] $\varphi_{\varepsilon}$ is smooth on $Y-Z_{\varepsilon}$ for a subvariety $Z_{\varepsilon}$;

  \item[(b)] $\varphi_{\varepsilon_{2}}\geqslant\varphi_{\varepsilon_{1}}\geqslant\varphi$ holds for any $0<\varepsilon_{1}\leqslant\varepsilon_{2}$;

  \item[(c)] $\mathscr{I}(\varphi)=\mathscr{I}(\varphi_{\varepsilon})$; and

  \item[(d)] $i\Theta_{L,\varphi_{\varepsilon}}\geqslant v-\varepsilon\omega$.
\end{enumerate}
\end{proposition}

Thanks to the openness property of the multiplier ideal sheaf \cite{GuZ15}, one can arrange $h_{\varepsilon}$ with logarithmic poles along $Z_{\varepsilon}$ according to the remark in \cite{DPS01}. Now we furthermore assume that there exists a section $\xi$ of some multiple $L^{k}$ such that $\sup_{Y}|\xi|_{h^{k}}<\infty$. Then the set 
\[
\{y\in Y;\nu(h_{\varepsilon},y)>0\}
\] 
for every $\varepsilon>0$ is contained in the subvariety $Z:=\{y\in Y;\xi(y)=0\}$ by property (b). Here $\nu(h_{\varepsilon},y)$ refers to the Lelong number \cite{Dem12} of $h_{\varepsilon}$ at $y$. Hence, instead of (a), we can assume that

(a') $h_{\varepsilon}$ is smooth on $Y-Z$ and has logarithmic poles along $Z$, where $Z$ is a subvariety of $Y$ independent of $\varepsilon$.

Now we are ready to prove Proposition \ref{p11}.
\begin{proof}[Proof of Proposition \ref{p11}]
Take a K\"{a}hler form $\Omega$ on $X$. By Proposition \ref{p31} and the remark after that, there exists a family of singular metrics $\{\varphi_{\varepsilon}\}$ with the following properties:
\begin{enumerate}
  \item[(a)] $\varphi_{\varepsilon}$ is smooth on $X^{\prime}$ for an open subvariety $X^{\prime}$ independent of $\varepsilon$;

  \item[(b)] $\varphi_{\varepsilon_{2}}\geqslant\varphi_{\varepsilon_{1}}\geqslant\varphi$ holds for any $0<\varepsilon_{1}\leqslant\varepsilon_{2}$;

  \item[(c)] $\mathscr{I}(\varphi)=\mathscr{I}(\varphi_{\varepsilon})$; and

  \item[(d)] $i\Theta_{\mathcal{O}_{E}(1),\varphi_{\varepsilon}}\geqslant v-\varepsilon\Omega$.
\end{enumerate}
Moreover, since $\xi\in\mathcal{O}_{E}(k)$, $\pi_{\ast}\xi\in S^{k}E$ by the canonical isomorphism
\[
\pi_{\ast}(\mathcal{O}_{E}(k))=S^{k}E. 
\]
As a result, 
\[
\begin{split}
\pi(X^{\prime})&=\pi(X-\{\xi=0\})\\
               &=Y-\{\pi_{\ast}\xi=0\}\\
               &=:Y^{\prime},
\end{split}
\]
which is an open subvariety of $Y$. In particular, it is easy to verify that $\varphi_{\varepsilon}\in\mathcal{H}(X)$ for every $\varepsilon$. Therefore it defines a singular metric 
\[
H^{l}_{\varepsilon}(U,U):=\int_{X_{y}}|U|^{2}e^{-l\varphi_{\varepsilon}}\frac{\omega^{r}_{\varphi_{\varepsilon},y}}{r!}
\]
on $S^{l}E$ (see Definition \ref{d22}) for every positive integer $l$ and $U\in S^{l}E$. Here we use the fact that $\pi_{\ast}\mathcal{O}_{E}(l)=S^{l}E$. It remains to prove the desired properties. 

(a), (c) and (d) are obvious. Recall that $S^{l}H$ can be rewritten as
\[
S^{l}H(U,U):=\int_{X_{y}}|u|^{2}e^{-l\varphi}\frac{\omega^{r}_{\varphi,y}}{r!}
\]
by \cite{Wu20b}. On the other hand, both $\varphi$ and $\varphi_{\varepsilon}$ are smooth on $X_{y}$ with $y\in Y^{\prime}$, we immediately conclude that $\varphi_{\varepsilon}$ converges locally uniformly and decreasingly to $\varphi$ on $X^{\prime}$. Hence $H^{l}_{\varepsilon}\rightarrow S^{l}H$ locally uniformly and increasingly on $Y^{\prime}$. (b) is proved. 

In the end, we prove (e) under the assumption that $(E,H)$ is strongly positive in the sense of Nakano. Indeed, by definition 
\[
i\Theta_{\mathcal{O}_{E}(1),\varphi}-\frac{i}{r+2}\pi^{\ast}\Theta_{\det E,\det H}
\]
is positive. Moreover, $i\Theta_{\mathcal{O}_{E}(1),\varphi}$ is also positive by Theorem 1.3 in \cite{Wu20b}. Hence 
\[
i\Theta_{\mathcal{O}_{E}(r+l+1),(r+l+1)\varphi}-i\pi^{\ast}\Theta_{\det E,\det H}
\]
is positive for all positive integer $l$. Consequently, 
\[
i\Theta_{\mathcal{O}_{E}(r+l+1),(r+l+1)\varphi_{\varepsilon}}-i\pi^{\ast}\Theta_{\det E,\det H}\geqslant-C_{l}\varepsilon\Omega.
\]
Now on $Y^{\prime\prime}$ we apply Berndtsson's curvature formula in \cite{Ber09,Ber11}, which says for any line bundle $(L,\psi)$ over $X$, the curvature of $\pi_{\ast}(K_{X/Y}\otimes L)$ associated with the $L^{2}$-metric satisfies that
\[
\sum<i\Theta_{i\bar{j}}s_{i},s_{j}>\geqslant\sum\int_{X_{y}}ic(\psi)_{i\bar{j}}s_{i}\wedge\bar{s}_{j}e^{-\psi}.
\]
Here $\{s_{i}\}$ is $n$-tuple of sections of $\pi_{\ast}(K_{X/Y}\otimes L)$. Let 
\[
(L,\psi)=(\mathcal{O}_{E}(r+l+1)\otimes\pi^{\ast}\det E^{\ast},(r+l+1)\varphi_{\varepsilon}-\pi^{\ast}\phi),
\]
where $\phi$ is the weight function of $\det H$. At this time, the direct image equals to $S^{l}E$, the $L^{2}$-metric is just $H^{l}_{\varepsilon}$ and the associated curvature satisfies 
\begin{equation}\label{e31}
\sum<i\Theta_{i\bar{j}}s_{i},s_{j}>\geqslant-C_{l}\varepsilon^{\prime}\sum\int_{X_{y}}ic(\Omega)_{i\bar{j}}s_{i}\wedge\bar{s}_{j}e^{-(r+l+1)\varphi_{\varepsilon}+\pi^{\ast}\phi},
\end{equation}
where $c(\Omega)_{i\bar{j}}$ is defined as (see Definition \ref{d22})
\[
\begin{split}
c(\Omega)_{i\bar{j}}:=<\frac{\delta}{\delta y_{i}},\frac{\delta}{\delta y_{j}}>_{\Omega}.
\end{split}
\]
The estimate (\ref{e31}) is due to formula (\ref{e22}), the fact that
\[
i\Theta_{\mathcal{O}_{E}(r+l+1),(r+l+1)\varphi_{\varepsilon}}-i\pi^{\ast}\Theta_{\det E,\phi}\geqslant-C_{l}\varepsilon\Omega
\]
and some elementary computation. Since $y\in Y^{\prime\prime}\subset Y^{\prime}$, $\varphi$ is smooth along $X_{y}$. Therefore
\[
\sum\int_{X_{y}}ic(\Omega)_{i\bar{j}}s_{i}\wedge\bar{s}_{j}e^{-(r+l+1)\varphi_{\varepsilon}+\pi^{\ast}\phi}
\]
is uniformly bounded with respect to $\varepsilon$. It obviously implies that
\[
  i\Theta_{S^{l}E,H^{l}_{\varepsilon}}\geqslant-C_{l}\varepsilon^{\prime\prime}\omega\otimes\mathrm{Id}_{S^{l}E}
  \]
in the sense of Nakano. Moreover, from the proof we see that if $i\Theta_{\mathcal{O}_{E}(1),\varphi}-\frac{i}{r+2}\pi^{\ast}\Theta_{\det E,\det H}>\varepsilon\omega$, which means that $(E,H)$ is strongly strictly positive in the sense of Nakano, we can even arrange the thing that
\[
i\Theta_{S^{l}E,H^{l}_{\varepsilon}}\geqslant C_{l}(1-\varepsilon^{\prime})\omega\otimes\mathrm{Id}_{S^{l}E}
\]
in the sense of Nakano.
\end{proof}

Then throughout the whole paper, for a given singular metric $H$ on $E$, $\varphi$ will always refer to its corresponding metric on $\mathcal{O}_{E}(1)$. $\{H^{k}_{\varepsilon}\}$ and $\{\varphi_{\varepsilon}\}$ will always be the regularising sequence provided by Proposition \ref{p11}. In particular, $H^{k}_{\varepsilon}$ is smooth on an open subvariety $Y^{\prime}$, whereas $\varphi_{\varepsilon}$ is smooth on $X^{\prime}$ with $X^{\prime}=\pi^{-1}(Y^{\prime})$. When $k=1$, $H^{1}_{\varepsilon}$ will be simply denoted by $H_{\varepsilon}$.

\section{Harmonic theory}
\subsection{Global theory}
Firstly, we use the method in \cite{Dem82} to construct a complete K\"{a}hler metric on $Y^{\prime}$ as follows. Since $Y^{\prime}$ is weakly pseudo-convex, we can take a smooth plurisubharmonic exhaustion function $\psi$ on $Y^{\prime}$. Define $\tilde{\omega}_{l}=\omega+\frac{1}{l}i\partial\bar{\partial}\psi^{2}$ for $l\gg0$. It is easy to verify that $\tilde{\omega}_{l}$ is a complete K\"{a}hler metric on $Y^{\prime}$ and $\tilde{\omega}_{l_{1}}\geqslant\tilde{\omega}_{l_{2}}\geqslant\omega$ for $l_{1}\leqslant l_{2}$. 

Let $L^{n,q}_{(2)}(Y^{\prime},E)_{H_{\varepsilon},\tilde{\omega}_{l}}$ be the $L^{2}$-space of $E$-valued $(n,q)$-forms $\alpha$ on $Y^{\prime}$ with respect to the inner product given by $H_{\varepsilon},\tilde{\omega}_{l}$. Then we have the orthogonal decomposition \cite{GHS98}
\begin{equation}\label{e41}
L^{n,q}_{(2)}(Y^{\prime},E)_{H_{\varepsilon},\tilde{\omega}_{l}}=\mathrm{Im}\bar{\partial}\bigoplus\mathcal{H}^{n,q}_{H_{\varepsilon}, \tilde{\omega}_{l}}(E)\bigoplus\mathrm{Im}\bar{\partial}^{\ast}_{H_{\varepsilon}}
\end{equation}
where
\[
\begin{split}
\mathrm{Im}\bar{\partial}&=\mathrm{Im}(\bar{\partial}:L^{n,q-1}_{(2)}(Y^{\prime},E)_{H_{\varepsilon},\tilde{\omega}_{l}}\rightarrow L^{n,q}_{(2)}(Y^{\prime},E)_{H_{\varepsilon},\tilde{\omega}_{l}}),\\
  \mathcal{H}^{n,q}_{H_{\varepsilon}, \tilde{\omega}_{l}}(E)&=\{\alpha\in L^{n,q}_{(2)}(Y^{\prime},E)_{H_{\varepsilon},\tilde{\omega}_{l}};\bar{\partial}\alpha=0, \bar{\partial}^{\ast}_{H_{\varepsilon}}\alpha=0\},
\end{split}
\]
and
\[
\mathrm{Im}\bar{\partial}^{\ast}_{H_{\varepsilon}}=\mathrm{Im}(\bar{\partial}^{\ast}_{H_{\varepsilon}}:L^{n,q+1}_{(2)}(Y^{\prime},E)_{H_{\varepsilon},\tilde{\omega}_{l}}\rightarrow L^{n,q}_{(2)}(Y^{\prime},E)_{H_{\varepsilon},\tilde{\omega}_{l}}).
\]
We give a brief explanation for decomposition (\ref{e41}). Usually $\mathrm{Im}\bar{\partial}$ is not closed in the $L^{2}$-space of a noncompact manifold even if the metric is complete. However, in the situation we consider here, $Y^{\prime}$ has the compactification $Y$, and the forms on $Y^{\prime}$ are bounded in $L^{2}$-norms. Such a form will have good extension properties. Therefore the set $L^{n,q}_{(2)}(Y^{\prime},E)_{H_{\varepsilon},\tilde{\omega}_{l}}\cap\mathrm{Im}\bar{\partial}$ behaves much like the space 
\[
\mathrm{Im}(\bar{\partial}:L^{n,q-1}_{(2)}(Y,E)_{H,\omega}\rightarrow L^{n,q}_{(2)}(Y,E)_{H,\omega}) 
\]
on $Y$, which is surely closed. The complete explanation can be found in \cite{Fuj12,Wu17}.

Now we have all the ingredients for the definition of $\Box_{0}$-harmonic forms. We denote the Lapalcian operator on $Y^{\prime}$ associated to $\tilde{\omega}_{l}$ and $H_{\varepsilon}$ by $\Box_{l,\varepsilon}$. Recall that for two $E$-valued $(n,q)$-forms $\alpha,\beta$ (not necessary to be $\bar{\partial}$-closed), we say they are cohomologically equivalent if there exits an $E$-valued $(n,q-1)$-form $\gamma$ such that $\alpha=\beta+\bar{\partial}\gamma$. We denote by $\alpha\in[\beta]$ this equivalence relationship.
\begin{definition}[=Definition \ref{d11}]\label{d41}
Let $\alpha$ be an $E$-valued $(n,q)$-form on $Y$ with bounded $L^{2}$-norm with respect to $H,\omega$. Assume that for every $l\gg1,\varepsilon\ll1$, there exists a representative $\alpha_{l,\varepsilon}\in[\alpha|_{Y^{\prime}}]$ such that
\begin{enumerate}
  \item $\Box_{l,\varepsilon}\alpha_{l,\varepsilon}=0$ on $Y^{\prime}$;
  \item $\alpha_{l,\varepsilon}\rightarrow\alpha|_{Y^{\prime}}$ in $L^{2}$-norm.
\end{enumerate}
Then we call $\alpha$ a $\Box_{0}$-harmonic form. The space of all the $\Box_{0}$-harmonic forms is denoted by
\[
\mathcal{H}^{n,q}(Y,E(H),\Box_{0}).
\]
\end{definition}

We will show that Definition \ref{d41} is compatible with the usual definition of $\Box_{0}$-harmonic forms for a smooth $H$ by proving the Hodge-type isomorphism, i.e. Proposition \ref{p12}.

\begin{proof}[Proof of Proposition \ref{p12}]
By Proposition \ref{p24}, we have 
\[
H^{q}(Y,K_{Y}\otimes E(H))\simeq\frac{\mathrm{Ker}(\bar{\partial}:A^{n,q}(Y,E(H))\rightarrow A^{n,q+1}(Y,E(H)))}{\mathrm{Im}(\bar{\partial}:A^{n,q-1}(Y,E(H))\rightarrow A^{n,q}(Y,E(H)))}.
\]
Hence a given cohomology class $[\alpha]\in H^{q}(Y,K_{Y}\otimes E(H))$ is represented by a $\bar{\partial}$-closed $E$-valued $(n,q)$-form $\alpha$ with $\|\alpha\|_{H,\omega}<\infty$. We denote $\alpha|_{Y^{\prime}}$ simply by $\alpha_{Y^{\prime}}$. Since $\tilde{\omega}_{l}\geqslant\omega$, it is easy to verify that
\[
|\alpha_{Y}|^{2}_{H_{\varepsilon},\tilde{\omega}_{l}}dV_{\tilde{\omega}_{l}}\leqslant|\alpha|^{2}_{H_{\varepsilon},\omega}dV_{\omega},
\]
which leads to inequality $\|\alpha_{Y}\|_{H_{\varepsilon},\tilde{\omega}_{l}}\leqslant\|\alpha\|_{H_{\varepsilon,\omega}}$ with $L^{2}$-norms. Hence by property (b), we have
$\|\alpha_{Y}\|_{H_{\varepsilon},\tilde{\omega}_{l}}\leqslant\|\alpha\|_{H,\omega}$ which implies
\[
\alpha_{Y}\in L^{n,q}_{(2)}(Y^{\prime},E)_{H_{\varepsilon},\tilde{\omega}_{l}}.
\]
By decomposition (\ref{e41}), we have a harmonic representative $\alpha_{l,\varepsilon}$ in
\[
\mathcal{H}^{n,q}_{H_{\varepsilon,\tilde{\omega}_{l}}}(E),
\]
which means that $\Box_{l,\varepsilon}\alpha_{l,\varepsilon}=0$ on $Y^{\prime}$ for all $l,\varepsilon$. Moreover, since a harmonic representative minimizes the $L^{2}$-norm, we have
\[
   \|\alpha_{l,\varepsilon}\|_{H_{\varepsilon},\tilde{\omega}_{l}}\leqslant\|\alpha_{Y}\|_{H_{\varepsilon},\tilde{\omega}_{l}}\leqslant\|\alpha\|_{H,\omega}.
\]
So there exists a limit $\tilde{\alpha}$ of (a subsequence of) $\{\alpha_{l,\varepsilon}\}$ such that
\[
\tilde{\alpha}\in[\alpha_{Y^{\prime}}].
\]
It is left to extend it to $Y$.

Indeed, if we admit Proposition \ref{p13} for the time being, $\ast\tilde{\alpha}$ will be a holomorphic $E$-valued $(n-q,0)$-form on $Y^{\prime}$. In particular, since the $\ast$-operator preserves the $L^{2}$-norm, the $L^{2}$-norm of $\ast\tilde{\alpha}$ is also bounded. Now apply the canonical $L^{2}$-extension theorem \cite{Ohs02}, we obtain a holomorphic extension of $\ast\tilde{\alpha}$ on $Y$, which is still denoted by $\ast\tilde{\alpha}$. We denote this morphism by
\[
\begin{split}
S^{q}:H^{q}(Y,K_{Y}\otimes E(H))&\rightarrow H^{0}(Y,\Omega^{n-q}_{Y}\otimes E)\\
[\alpha]&\mapsto\ast\tilde{\alpha}.
\end{split}
\]
Let $\hat{\alpha}:=c_{n-q}\omega^{q}\wedge\ast\tilde{\alpha}$, we then obtain an $E$-valued $(n,q)$-form on $Y$. It is easy to verify that $\hat{\alpha}|_{Y^{\prime}}=\tilde{\alpha}$. In summary, we have successfully defined a morphism
\[
\begin{split}
i:H^{q}(Y,K_{Y}\otimes E(H))&\rightarrow\mathcal{H}^{n,q}(Y,E(H),\Box_{0})\\
[\alpha]&\mapsto\hat{\alpha}.
\end{split}
\] 

On the other hand, for a given $\alpha\in\mathcal{H}^{n,q}(Y,E(H),\Box_{0})$, by definition there exists an $\alpha_{l,\varepsilon}\in[\alpha_{Y^{\prime}}]$ with $\alpha_{l,\varepsilon}\in\mathcal{H}^{n,q}_{H_{\varepsilon,\tilde{\omega}_{l}}}(E)$ for every $l,\varepsilon$ such that $\lim\alpha_{l,\varepsilon}=\alpha_{Y^{\prime}}$. In particular, $\bar{\partial}\alpha_{l,\varepsilon}=0$. So all of the $\alpha_{l,\varepsilon}$ together with $\alpha_{Y^{\prime}}$ define a common cohomology class $[\alpha_{Y^{\prime}}]$ in $H^{n,q}(Y^{\prime},E(H))$. It is left to extend this class to $Y$.

We use the $S^{q}$ again. It maps $[\alpha_{Y^{\prime}}]$ to
\[
S^{q}(\alpha_{Y^{\prime}})\in H^{0}(Y,\Omega^{n-q}_{Y}\otimes E).
\]
Then
\[
   c_{n-q}\omega^{q}\wedge S^{q}(\alpha_{Y^{\prime}})\in H^{q}(Y,K_{Y}\otimes E(H))
\]
with $[(c_{n-q}\omega^{q}\wedge S^{q}(\alpha_{Y^{\prime}}))|_{Y^{\prime}}]=[\alpha_{Y^{\prime}}]$ as a cohomology class in
\[
H^{n,q}(Y^{\prime},E(H)).
\]
We denote this morphism by 
\[
\begin{split}
j:\mathcal{H}^{n,q}(Y,E(H),\Box_{0})&\rightarrow H^{q}(Y,K_{Y}\otimes E(H))\\
\alpha&\mapsto[c_{n-q}\omega^{q}\wedge S^{q}(\alpha_{Y})].
\end{split}
\]
It is easy to verify that $i\circ j=\textrm{Id}$ and $j\circ i=\textrm{Id}$. The proof is finished.
\end{proof}

We now prove Proposition \ref{p13} to finish this subsection.
\begin{proof}[Proof of Proposition \ref{p13}]
(1) Since $\alpha$ is $\Box_{0}$-harmonic, there exists an $\alpha_{l,\varepsilon}\in[\alpha_{Y^{\prime}}]$ with $\alpha_{l,\varepsilon}\in\mathcal{H}^{n,q}_{H_{\varepsilon,\tilde{\omega}_{l}}}(E)$ for every $l,\varepsilon$ such that $\lim\alpha_{l,\varepsilon}=\alpha_{Y^{\prime}}$. In particular, $\bar{\partial}\alpha_{l,\varepsilon}=\bar{\partial}^{\ast}_{H_{\varepsilon}}\alpha_{l,\varepsilon}=0$. Apply Proposition \ref{p25} on $Y^{\prime}$ with $\eta=1$, we have
\begin{equation}\label{e42}
\begin{split}
0=&\|\bar{\partial}\alpha_{l,\varepsilon}\|^{2}_{H_{\varepsilon},\tilde{\omega}_{l}}+\|\bar{\partial}^{\ast}_{H_{\varepsilon}}\alpha_{l,\varepsilon}\|^{2}_{H_{\varepsilon},\tilde{\omega}_{l}}\\
=&\|\partial^{\ast}_{H_{\varepsilon}}\alpha_{l,\varepsilon}\|^{2}_{H_{\varepsilon},\tilde{\omega}_{l}}+\|\partial_{H_{\varepsilon}}\alpha_{l,\varepsilon}\|^{2}_{H_{\varepsilon},\tilde{\omega}_{l}}+<i[\Theta_{E,H_{\varepsilon}},\Lambda]\alpha_{l,\varepsilon},\alpha_{l,\varepsilon}>_{H_{\varepsilon},\tilde{\omega}_{l}}.
\end{split}
\end{equation}
Remember that $i\Theta_{E,H_{\varepsilon}}\geqslant-\varepsilon\omega\otimes\mathrm{Id}_{E}$ in the sense of Nakano,
\[
<i[\Theta_{E,H_{\varepsilon}},\Lambda]\alpha_{l,\varepsilon},\alpha_{l,\varepsilon}>_{H_{\varepsilon},\tilde{\omega}_{l}}\geqslant-\varepsilon^{\prime}\tilde{\omega}_{l}
\]
by elementary computation. Now take the limit on the both sides of formula (\ref{e42}) with respect to $l,\varepsilon$, we eventually obtain that
\[
\begin{split}
&\lim\|\partial^{\ast}_{H_{\varepsilon}}\alpha_{l,\varepsilon}\|^{2}_{H_{\varepsilon},\tilde{\omega}_{l}}=\lim\|\partial_{H_{\varepsilon}}\alpha_{l,\varepsilon}\|^{2}_{H_{\varepsilon},\tilde{\omega}_{l}}\\
&=\lim<i[\Theta_{E,H_{\varepsilon}},\Lambda]\alpha_{l,\varepsilon},\alpha_{l,\varepsilon}>_{H_{\varepsilon},\tilde{\omega}_{l}}=0.
\end{split}
\]
In particular,
\[
0=\lim\partial^{\ast}_{H_{\varepsilon}}\alpha_{l,\varepsilon}=\ast\bar{\partial}\ast\lim\alpha_{l,\varepsilon}=\ast\bar{\partial}\ast\alpha
\]
in $L^{2}$-topology. Equivalently, $\bar{\partial}\ast\alpha=0$ on $Y^{\prime}$ in analytic topology, hence is a holomorphic $E$-valued $(n-q,0)$-form on $Y^{\prime}$. On the other hand, since $\ast\alpha$ has the bounded $L^{2}$-norm on $Y^{\prime}$, it extends to the whole space by classic $L^{2}$-extension theorem \cite{Ohs02}. In other word, $\bar{\partial}\ast\alpha=0$ actually holds on $Y$, hence $\ast\alpha$ is an $E$-valued holomorphic $(n-q,0)$-form on $Y$. 

(2) Apply the same argument before, we have $\alpha=c_{n-q}\omega^{q}\wedge\ast\alpha$. Therefore $\alpha$ must be smooth. 
\end{proof}

\subsection{Local theory}
In practice, we will also deal with the manifold with boundary, hence the local theory is also needed. Let $V$ be a bounded domain with smooth boundary $\partial V$ on $(Y,\omega)$. Moreover, there is a smooth plurisubharmonic exhaustion function $r$ on $V$. In particular, $V=\{r<0\}$ and $dr\neq0$ on $\partial V$. The volume form $dS$ of the real hypersurface $\partial V$ is defined by $dS:=\ast(dr)/|dr|_{\omega}$. Let $G$ be a smooth Hermitian metric on $E$. Let $L^{p,q}_{(2)}(V,E)_{G,\omega}$ be the space of $E$-valued $(p,q)$-forms on $V$ which are $L^{2}$-bounded with respect to $G,\omega$. Setting $\tau:=dS/|dr|_{\omega}$ we define the inner product on $\partial V$ by
\[
[\alpha,\beta]_{G}:=\int_{\partial V}<\alpha,\beta>_{G}\tau
\] 
for $\alpha,\beta\in L^{p,q}_{(2)}(V,E)_{G,\omega}$. Then by Stokes' theorem we have the following:
\begin{equation}\label{e43}
\begin{split}
<\bar{\partial}\alpha,\beta>_{G}&=<\alpha,\bar{\partial}^{\ast}\beta>_{G}+[\alpha,e(\bar{\partial}r)^{\ast}\beta]_{G}\\
<\partial_{G}\alpha,\beta>_{G}&=<\alpha,\partial^{\ast}_{G}\beta>_{G}+[\alpha,e(\partial r)^{\ast}\beta]_{G}
\end{split}
\end{equation}
where $\bar{\partial}^{\ast},\partial^{\ast}$ are the adjoint operators defined on $Y$ (see Sect.2.6). In particular, if $e(\bar{\partial}r)^{\ast}\alpha=0$, the Bochner formula (Proposition \ref{p25}) on $V$ will be
\begin{equation}\label{e44}
\begin{split}
&\Box_{G}=\bar{\Box}_{G}+[i\Theta_{E,G},\Lambda],\\
&\Box_{e^{-\chi}G}=\bar{\Box}_{e^{-\chi}G}+[i(\Theta_{E,G}+\partial\bar{\partial}\chi),\Lambda]\textrm{ and }\\
&\|\sqrt{\eta}(\bar{\partial}+e(\bar{\partial}\chi))\alpha\|^{2}_{G,\omega}+\|\sqrt{\eta}\bar{\partial}^{\ast}\alpha\|^{2}_{G,\omega}=\|\sqrt{\eta}(\partial^{\ast}_{G}-e(\partial\chi)^{\ast})\alpha)\|^{2}_{G,\omega}\\
&+\|\sqrt{\eta}\partial_{G}\alpha\|^{2}_{G,\omega}+<\eta i[\Theta_{E,G}+\partial\bar{\partial}\varphi,\Lambda]\alpha,\alpha>_{G,\omega}+[\partial^{\ast}_{G}\alpha,e(\partial r)^{\ast}\alpha]_{G}
\end{split}
\end{equation} 
where $\eta$ is a positive smooth function on $Y$ with $\chi:=\log\eta$.

We then define the space of harmonic forms on $V$ by
\[
\mathcal{H}^{n,q}(V,E(G),r,\omega):=\{\alpha\in L^{n,q}_{(2)}(\bar{V},E)_{G,\omega};\bar{\partial}\alpha=\bar{\partial}^{\ast}\alpha=e(\bar{\partial}r)^{\ast}\alpha=0\}.
\]

Now return back to our setting: $(E,H)$ is a singular Hermitian vector bundle that is strongly positive in the sense of Nakano. By Proposition \ref{p11}, there exists a regularising sequence $\{H_{\varepsilon}\}$ such that
\[
i\Theta_{E,H_{\varepsilon}}\geqslant-\varepsilon\omega\otimes\mathrm{Id}_{E}
\] 
in the sense of Nakano. Take $V\subset Y^{\prime}$. Using the same notations as in Sect.4.1, the harmonic space with respect to $H$ is defined as
\[
\begin{split}
\mathcal{H}^{n,q}(V,E(H),r):=&\{\alpha\in L^{n,q}_{(2)}(V,E)_{H,\omega};\textrm{there exits }\alpha_{l,\varepsilon}\in[\alpha]\textrm{ such that }\\
&\alpha_{l,\varepsilon}\in\mathcal{H}^{n,q}(V,E(G_{\varepsilon}),r,\tilde{\omega}_{l})\textrm{ and }\alpha_{l,\varepsilon}\rightarrow\alpha\textrm{ in }L^{2}\textrm{-limit}\}.
\end{split}
\]
We then generalise the work in \cite{Tak95} here.
\begin{proposition}\label{p41}
Assume that $(E,H)$ is strongly positive in the sense of Nakano (so that $E(H)$ is coherent by Proposition \ref{p23}). Then we have the following conclusions:
\begin{enumerate}
  \item Assume $\alpha\in L^{n,q}_{(2)}(Y,E)_{H,\omega}$ satisfied $e(\bar{\partial}r)^{\ast}\alpha=0$ on $V$. Then $\alpha$ satisfies $\bar{\partial}\alpha=\lim\bar{\partial}^{\ast}_{H_{\varepsilon}}\alpha=0$ on $V$ if and only if $\bar{\partial}\ast\alpha=0$ and $\lim<ie(\Theta_{H_{\varepsilon}}+\partial\bar{\partial}r)\Lambda\alpha,\alpha>_{H_{\varepsilon}}=0$ on $V$.
  \item $\mathcal{H}^{n,q}(V,E(H),r)$ is independent of the choice of exhaustion function $r$.
  \item $\mathcal{H}^{n,q}(V,E(H),r)\simeq H^{q}(V,K_{Y}\otimes E(H))$.
  \item For Stein open subsets $V_{1},V_{2}$ in $V$ such that $V_{2}\subset V_{1}$, the restriction map
  \[
  \mathcal{H}^{n,q}(V_{1},E(H),r)\rightarrow\mathcal{H}^{n,q}(V_{2},E(H),r)
  \]
  is well-defined, and further it satisfies the following commutative diagram:
  \[
  \xymatrix{
  \mathcal{H}^{n,q}(V_{1},E(H),r) \ar[d]_{i^{1}_{2}} \ar[r]^{S^{q}_{V_{1}}} & H^{0}(V_{1},\Omega^{n-q}_{Y}\otimes E(H)) \ar[d] \\
  \mathcal{H}^{n,q}(V_{2},E(H),r) \ar[r]^{S^{q}_{V_{2}}} & H^{0}(V_{2},\Omega^{n-q}_{Y}\otimes E(H)).      
  }
  \]
\end{enumerate}
\begin{proof}
The proof uses the same argument as Theorems 4.3 and 5.2 in \cite{Tak95} with minor adjustment. So we only provide the necessary details.

(1) Let $G=e^{-r}H$ and $G_{\varepsilon}=e^{-r}H_{\varepsilon}$. If $\bar{\partial}\alpha=\lim\bar{\partial}^{\ast}_{H_{\varepsilon}}\alpha=0$, then $\lim\bar{\partial}^{\ast}_{G_{\varepsilon}}\alpha=0$ and so $\lim\Box_{G_{\varepsilon}}\alpha=0$. By formula (\ref{e44}) we obtain
\[
\lim(\|\partial^{\ast}_{H_{\varepsilon}}\alpha\|^{2}_{G_{\varepsilon}}+<ie(\Theta_{E,G_{\varepsilon}}+\partial\bar{\partial}r)\Lambda\alpha,\alpha>_{G_{\varepsilon}}+[ie(\partial\bar{\partial}r)\Lambda\alpha,\alpha]_{G_{\varepsilon}})=0
\]
on $V$. Since $<ie(\Theta_{E,H_{\varepsilon}}+\partial\bar{\partial}r)\Lambda\alpha,\alpha>_{G_{\varepsilon}}\geqslant-\varepsilon\omega$ and 
\[
[ie(\partial\bar{\partial}r)\Lambda\alpha,\alpha]_{G_{\varepsilon}}\geqslant0,
\] 
the equality above implies that
\[
\ast\bar{\partial}\ast\alpha=\lim<ie(\Theta_{E,H_{\varepsilon}}+\partial\bar{\partial}r)\Lambda\alpha,\alpha>_{G_{\varepsilon}}=\lim[ie(\partial\bar{\partial}r)\Lambda\alpha,\alpha]_{G_{\varepsilon}}=0.
\]
Equivalently, 
\[
\bar{\partial}\ast\alpha=\lim<ie(\Theta_{E,H_{\varepsilon}}+\partial\bar{\partial}r)\Lambda\alpha,\alpha>_{H_{\varepsilon}}=0.
\]
The necessity is proved. 

Now assiume that $\bar{\partial}\ast\alpha=0$ and $\lim<ie(\Theta_{H_{\varepsilon}}+\partial\bar{\partial}r)\Lambda\alpha,\alpha>_{H_{\varepsilon}}=0$. Since $r$ is plurisubharmonic and $\lim<ie(\Theta_{H_{\varepsilon}})\Lambda\alpha,\alpha>_{H_{\varepsilon}}\geqslant0$, we have $\lim<ie(\partial\bar{\partial}r)\Lambda\alpha,\alpha>_{H_{\varepsilon}}=\lim<ie(\Theta_{H_{\varepsilon}})\Lambda\alpha,\alpha>_{H_{\varepsilon}}=0$. By formula (\ref{e44}) we have $\bar{\partial}\alpha=\lim\bar{\partial}^{\ast}_{H_{\varepsilon}}\alpha=0$.

(2) Let $\tau$ be an arbitrary smooth plurisubharmonic function on $V$. Donnelly and Xavier's formula (\ref{e25}) implies that $\bar{\partial}e(\bar{\partial}\tau)^{\ast}\alpha=ie(\partial\bar{\partial})\Lambda\alpha$ if 
\[
\alpha\in\mathcal{H}^{n,q}(V,E(H),r).
\]
Therefore
\[
\begin{split}
<ie(\partial\bar{\partial}\tau)\Lambda\alpha,\alpha>_{e^{\tau}H_{\varepsilon}}&=<\bar{\partial}e(\bar{\partial}\tau)^{\ast}\alpha,\alpha>_{e^{\tau}H_{\varepsilon}}\\
&=<e(\bar{\partial}\tau)^{\ast}\alpha,\bar{\partial}^{\ast}_{e^{\tau}H_{\varepsilon}}\alpha>_{e^{\tau}H_{\varepsilon}}\\
&=<e(\bar{\partial}\tau)^{\ast}\alpha,\bar{\partial}^{\ast}_{H_{\varepsilon}}\alpha>_{e^{\tau}H_{\varepsilon}}-\|e(\bar{\partial}\tau)^{\ast}\alpha\|^{2}_{e^{\tau}H_{\varepsilon}}.
\end{split}
\]
Take the limit with respect to $\varepsilon$, we then obtain that
\[
<ie(\partial\bar{\partial}\tau)\Lambda\alpha,\alpha>_{e^{\tau}H}=-\|e(\bar{\partial}\tau)^{\ast}\alpha\|^{2}_{e^{\tau}H}.
\]
Notice that $\tau$ is plurisubharmonic, we actually have 
\[
<ie(\partial\bar{\partial}\tau)\Lambda\alpha,\alpha>_{e^{\tau}H}=\|e(\bar{\partial}\tau)^{\ast}\alpha\|^{2}_{e^{\tau}H}=0.
\]
Combine with (1), we eventually obtain that
\[
\mathcal{H}^{n,q}(V,E(H),r)=\mathcal{H}^{n,q}(V,E(H),r+\tau)
\]
for any smooth plurisubharmonic $\tau$, hence the desired conclusion.

(3) is similar with Proposition \ref{p12}, and we omit its proof here.

(4) is intuitive due to the discussions in the global setting. In particular, $S^{q}_{V_{i}}$ with $i=1,2$ is defined in the proof of Proposition \ref{p12}.
\end{proof}
\end{proposition}

\section{The main theorem}
This section is devoted to prove Theorem \ref{t11}.
\begin{theorem}[=Theorem \ref{t11}]\label{t51}
Let $f:Y\rightarrow Z$ be a fibration between two compact K\"{a}hler manifolds. Let $n=\dim Y$ and $m=\dim Z$. Suppose that $(E,H)$ is a (singular) Hermitian vector bundle over $Y$ that is strongly positive in the sense of Nakano. Moreover, assume that $H|_{Y_{z}}$ is well-defined for every $z\in Z$. Then the following theorems hold:

{\rm \textbf{I Decomposition Theorem}}. The Leray spectral sequence \cite{GrH78}
  \[
  E^{p,q}_{2}=H^{p}(Z,R^{q}f_{\ast}(K_{Y}\otimes E(H)))\Rightarrow H^{p+q}(Y,K_{Y}\otimes E(H))
  \]
  degenerates at $E_{2}$. As a consequence, it holds that
  \[
  \dim H^{i}(Y,K_{Y}\otimes E(H))=\sum_{p+q=i}\dim H^{p}(Z,R^{q}f_{\ast}(K_{Y}\otimes E(H)))
  \]
  for any $i\geqslant0$.
  
{\rm \textbf{II Torsion freeness Theorem}}. For $q\geqslant0$ the sheaf homomorphism
  \[
  L^{q}: f_{\ast}\Omega^{n-q}\otimes E(H)\rightarrow R^{q}f_{\ast}(K_{Y}\otimes E(H))
  \]
  induced by the $q$-times left wedge product by $\omega$ admits a splitting sheaf homomorphism
  \[
  S^{q}: R^{q}f_{\ast}(K_{Y}\otimes E(H))\rightarrow f_{\ast}\Omega^{n-q}\otimes E(H)\textrm{ with }L^{q}\circ S^{q}=\mathrm{id}.
  \]
  In particular, $R^{q}f_{\ast}(K_{Y}\otimes E(H))$ is torsion free \cite{Kob87} for $q\geqslant0$ and vanishes if $q>n-m$. Furthermore, it is even reflexive if $E(H)=E$.
  
{\rm \textbf{III Injectivity Theorem}}. Let $(L,h)$ be a (singular) Hermitian line bundle over $Y$. Recall that $Y^{\prime}$ is the open subvariety appeared in Proposition \ref{p11}. Assume the following conditions:
\begin{enumerate}
  \item[(a)] the singular part of $h$ is contained in $Y-Y^{\prime}$;
  \item[(b)] $i\Theta_{L,h}\geqslant\gamma$ for some real smooth $(1,1)$-form $\gamma$ on $Y$;
  \item[(c)] for some rational $\delta\ll1$, the $\mathbb{Q}$-twisted bundle
  \[
  E<-\delta L>|_{Y_{z}}
  \]
  is strongly positive in the sense of Nakano for every $z$.
\end{enumerate}
For a (non-zero) section $s$ of $L$ with $\sup_{Y}|s|_{h}<\infty$, the multiplication map induced by the tensor product with $s$
  \[
 R^{q}f_{\ast}(s): R^{q}f_{\ast}(K_{Y}\otimes E(H))\rightarrow R^{q}f_{\ast}(K_{Y}\otimes(E\otimes L)(H\otimes h))
  \]
  is well-defined and injective for any $q\geqslant0$.
  
{\rm \textbf{IV Relative vanishing Theorem}}. Let $g:Z\rightarrow W$ be a fibration to a compact K\"{a}hler manifold $W$. Then the Leray spectral sequence:
  \[
  R^{p}g_{\ast}R^{q}f_{\ast}(K_{Y}\otimes E(H))\Rightarrow R^{p+q}(g\circ f)_{\ast}(K_{Y}\otimes E(H))
  \]  
  degenerates.
\begin{proof}
\textbf{I}. Let $\{U,r_{U}\}$ be a finite Stein covering of $Z$ with smooth strictly plurisubharmonic exhaustion function $r_{U}$. Let 
\[
\mathcal{H}^{n,q}(f^{-1}(U),E(H),f^{\ast}r_{U})
\] 
be the harmonic space defined in Sect.4.2. Then the data 
\[
\{\mathcal{H}^{n,q}(f^{-1}(U),E(H),f^{\ast}r_{U}),i^{1}_{2}\}
\] 
with the restriction morphisms 
\[
i^{1}_{2}:\mathcal{H}^{n,q}(f^{-1}(U_{1}),E(H),f^{\ast}r_{U_{1}})\rightarrow\mathcal{H}^{n,q}(f^{-1}(U_{2}),E(H),f^{\ast}r_{U_{2}}),
\] 
$(U_{2},r_{U_{2}})\subset(U_{1},r_{U_{1}})$, yields a presheaf \cite{Har77} on $Z$ by Proposition \ref{p41}, (4). We denote the associated sheaf by $f_{\ast}\mathcal{H}^{n,q}(E(H))$. Since 
\[
R^{q}f_{\ast}(K_{Y}\otimes E(H))
\] 
is defined as the sheaf associated with the presheaf
\[
U\rightarrow H^{q}(f^{-1}(U),K_{Y}\otimes E(H)),
\]
the sheaf $f_{\ast}\mathcal{H}^{n,q}(E(H))$ is isomorphic to $R^{q}f_{\ast}(K_{Y}\otimes E(H))$ by Proposition \ref{p41}, (3).

Let $\mathcal{C}^{p,q}$ be the space of $p$-cochains associated to $\{U\}$ with values in $f_{\ast}\mathcal{H}^{n,q}(E(H))$. Then $\{\mathcal{C}^{p,q},\delta\}$ is a complex with the coboundary operator $\delta$ whose cohomology group $H^{p}(\mathcal{C}^{\ast,q})$ is isomorphic to 
\[
E^{p,q}_{2}=H^{p}(Y,R^{q}f_{\ast}(K_{Y}\otimes E(H))).
\] 
Since the differential $d:=\delta+\bar{\partial}$ is identically zero, (recall that an element $\alpha\in f_{\ast}\mathcal{H}^{n,q}(E(H))$ must satisfy $\bar{\partial}\alpha=0$), $d_{2}:E^{p,q}_{2}\rightarrow E^{p+2,q-1}_{2}$ is also so which implies the degeneration of the Leray spectral sequence at $E_{2}$, i.e. $E^{p,q}_{2}=E^{p,q}_{\infty}$. 

\textbf{II}. Recall that in the proof of Proposition \ref{p12}, we've defined two morphisms
\[
\begin{split}
S^{q}:H^{q}(Y,K_{Y}\otimes E(H))&\rightarrow H^{0}(Y,\Omega^{n-q}_{Y}\otimes E(H))\\
[\alpha]&\mapsto\ast\tilde{\alpha}
\end{split}
\]
and
\[
\begin{split}
L^{q}:H^{0}(Y,\Omega^{n-q}_{Y}\otimes E(H))&\rightarrow H^{q}(Y,K_{Y}\otimes E(H))\\
\beta&\mapsto[c_{n-q}\omega^{q}\wedge\beta]
\end{split}
\]
such that $L^{q}\circ S^{q}=\textrm{Id}$. These two morphisms lift to the direct images as
\[
\begin{split}
S^{q}:R^{q}f_{\ast}(K_{Y}\otimes E(H))&\rightarrow f_{\ast}(\Omega^{n-q}_{Y}\otimes E(H)),\\
L^{q}:f_{\ast}(\Omega^{n-q}_{Y}\otimes E(H))&\rightarrow R^{q}f_{\ast}(K_{Y}\otimes E(H)).
\end{split}
\]
Here we abuse the notation. In particular, $L^{q}\circ S^{q}=\textrm{Id}$. As a result, $R^{q}f_{\ast}(K_{Y}\otimes E(H))$ is splitting embedded into $f_{\ast}(\Omega^{n-q}_{Y}\otimes E(H))$. Obviously, $f_{\ast}(\Omega^{n-q}_{Y}\otimes E(H))$ is an $\mathcal{O}_{Z}$-submodule of $f_{\ast}(\Omega^{n-q}_{Y}\otimes E)$, whereas $f_{\ast}(\Omega^{n-q}_{Y}\otimes E)$ is torsion free (even reflexive) by \cite{Har80}. We then conclude that $R^{q}f_{\ast}(K_{Y}\otimes E(H))$, as an $\mathcal{O}_{Z}$-submodule of $f_{\ast}(\Omega^{n-q}_{Y}\otimes E)$, is also torsion free. 

When $E(H)=E$, the reflexivity is also inherited since it is a splitting embedding.

\textbf{III}. It is enough to prove that for an arbitrary point $z\in Z$,
\[
\otimes s:H^{q}(f^{-1}(z),K_{Y}\otimes E(H))\rightarrow H^{q}(f^{-1}(z),K_{Y}\otimes(E\otimes L)(H\otimes h))
\] 
is injective. Equivalently, we should prove that 
\[
\otimes s:\mathcal{H}^{n,q}(f^{-1}(z),E(H))\rightarrow\mathcal{H}^{n,q}(f^{-1}(z),(E\otimes L)(H\otimes h))
\]
is well-defined and injective by Proposition \ref{p12}.

Let $\alpha\in\mathcal{H}^{n,q}(f^{-1}(z),E(H))$, by Proposition \ref{p41}, (1) with $r=1$, we have $\bar{\partial}\ast\alpha=\lim<ie(\Theta_{H_{\varepsilon}})\Lambda\alpha,\alpha>_{H_{\varepsilon}}=0$. Let $\{h_{\varepsilon}\}$ be a regularising sequence of $h$ (Proposition \ref{p11}), hence $\{H_{\varepsilon}\otimes h_{\varepsilon}\}$ is a regularising sequence of $H\otimes h$. In particular, $H_{\varepsilon}\otimes h_{\varepsilon}$ is smooth on $Y^{\prime}\cap f^{-1}(z)$. Therefore
\[
\bar{\partial}\ast(s\alpha)=s\bar{\partial}\ast\alpha=0
\]
and on $Y^{\prime}\cap f^{-1}(z)$ we have
\[
\begin{split}
&\lim<ie(\Theta_{E\otimes L,H_{\varepsilon}\otimes h_{\varepsilon}})\Lambda s\alpha,s\alpha>_{H_{\varepsilon}\otimes h_{\varepsilon}}\\
=&\lim|s|^{2}_{h_{\varepsilon}}\lim<ie(\Theta_{E\otimes L,H_{\varepsilon}\otimes h_{\varepsilon}})\Lambda\alpha,\alpha>_{H_{\varepsilon}}\\
=&|s|^{2}_{h}\lim<ie(\Theta_{L,h_{\varepsilon}}\otimes\mathrm{Id}_{E})\Lambda\alpha,\alpha>_{H_{\varepsilon}}\\
\leqslant&\frac{1}{\delta}|s|^{2}_{h}\lim<ie(\Theta_{E,H_{\varepsilon}})\Lambda\alpha,\alpha>_{H_{\varepsilon}}\\
=&0.
\end{split}
\]
The inequality is due to the assumption that $E<-\delta L>$ is strongly positive in the sense of Nakano. Now we apply the Bochner formula (Proposition \ref{p25}) on $(E\otimes L, H_{\varepsilon}\otimes h_{\varepsilon})$ and take the limit, eventually we obtain that
\[
\lim\|\bar{\partial}^{\ast}_{H_{\varepsilon}\otimes h_{\varepsilon}}(s\alpha)\|^{2}_{H_{\varepsilon}\otimes h_{\varepsilon}}=0,
\]
and
\[
\bar{\partial}(s\alpha)=0.
\]
On the other hand, in any local coordinate neighbourhood $V$ in $f^{-1}(z)$,
\[
\int_{V}|s\alpha|^{2}_{H\otimes h}\leqslant\sup_{Y}|s|^{2}_{h}\int_{V}|\alpha|^{2}_{H}<\infty.
\]
The discussions above together imply that 
\[
s\alpha\in \mathcal{H}^{n,q}(f^{-1}(z),(E\otimes L)(H\otimes h))
\] 
by Proposition \ref{p41}, (1). Therefore 
\[
\otimes s:\mathcal{H}^{n,q}(f^{-1}(z),E(H))\rightarrow\mathcal{H}^{n,q}(f^{-1}(z),(E\otimes L)(H\otimes h))
\]
is well-defined. The injectivity is obvious. 

\textbf{IV} is a direct application of \textbf{I}. For a given point $w\in W$, by \textbf{I} the Leray spectral sequence 
\[
E^{p,q}_{2}=H^{p}(g^{-1}(w),R^{q}f_{\ast}(K_{Y}\otimes E(H)))\Rightarrow H^{p+q}((g\circ f)^{-1}(w),K_{Y}\otimes E(H))
\]
degenerates at $E_{2}$. Therefore the Leray spectral sequence:
  \[
  R^{p}g_{\ast}R^{q}f_{\ast}(K_{Y}\otimes E(H))\Rightarrow R^{p+q}(g\circ f)_{\ast}(K_{Y}\otimes E(H))
  \]  
  degenerates.
\end{proof} 
\end{theorem}

\section{Vanishing theorem}
We should prove Theorem \ref{t12} in the end.
\begin{theorem}[=Theorem \ref{t12}]\label{t61}
Let $f:Y\rightarrow Z$ be a fibration between two compact K\"{a}hler manifolds. 

{\rm \textbf{I Nadel-type vanishing Theorem}}. Let $(L,h)$ be an $f$-big line bundle, and let $(E,H)$ be a vector bundle that is strongly positive in the sense of Nakano. Assume that $H|_{Y_{z}}$ is well-defined for every $z$. Then
\[
R^{q}f_{\ast}(K_{Y}\otimes(E\otimes L)(H\otimes h))=0\textrm{ for every }q>0.
\]

{\rm \textbf{II Nakano-type vanishing Theorem}}. Assume that $(E,H)$ is strongly strictly positive in the sense of Nakano. Then
\[
H^{q}(Y,K_{Y}\otimes S^{l}E(S^{l}H))=0\textrm{ for every }l,q>0.
\] 

{\rm \textbf{III Griffiths-type vanishing Theorem}}. Assume that $(E,H)$ is strictly positive in the sense of Griffiths. Then
\[
H^{q}(Y,K_{Y}\otimes S^{l}(E\otimes\det E)(S^{l}(H\otimes\det H)))=0\textrm{ for every }l,q>0.
\] 
\end{theorem}
Recall that $L$ is $f$-big if the Iitaka dimension \cite{Fuj13} 
\[
\kappa(Y_{z},L)=\dim Y-\dim Z
\]  
for every $z\in Z$.  
\begin{proof}
\textbf{I}. We claim that if $L$ is $f$-big, there exists a singular metric $h$ on $L$ such that $h|_{Y_{z}}$ is well-defined and $i\Theta_{L,h}|_{Y_{z}}$ is strictly positive for every $z\in Z$. 

Indeed, by definition for every $z\in Z$ there exists a singular metric $h_{z}$ on $L|_{Y_{z}}$ such that $i\Theta_{L,h_{z}}$ is strictly positive \cite{Dem12}. Now take a smooth metric $h_{0}$ on $L$ and define a singular metric on $L$ as follows:
\[
h:=h^{\delta}_{0}\otimes h^{1-\delta}_{z}\textrm{ for }y\in Y_{z}.
\]
It is easy to verify that $h$ satisfies the desired property when $\delta$ is small enough. In particular, $i\Theta_{L,h}\geqslant\gamma$ for some real smooth $(1,1)$-form on $Y$.																																
Let $A$ be a sufficiently ample line bundle over $Z$. Then
\[
H^{q}(Y_{z},K_{Y}\otimes(E\otimes L)(H\otimes h)\otimes A)=0\textrm{ for }q>0
\] 
by Serre's asymptotic vanishing theorem \cite{Har77}, hence
\[
R^{q}f_{\ast}(K_{Y}\otimes(E\otimes L)(H\otimes h)\otimes A)=0\textrm{ for }q>0.
\]

On the other hand, it is easy to verify that $A$ satisfies the conditions in Theorem \ref{t11}, III, therefore
\[
R^{q}f_{\ast}(K_{Y}\otimes(E\otimes L)(H\otimes h))\rightarrow R^{q}f_{\ast}(K_{Y}\otimes(E\otimes L)(H\otimes h)\otimes A)
\]
is injective. As a result, $R^{q}f_{\ast}(K_{Y}\otimes(E\otimes L)(H\otimes h))=0$ for $q>0$. 

\textbf{II}. By Proposition \ref{p12}, it is enough to prove that
\[
\mathcal{H}^{n,q}(Y,S^{l}E(S^{l}H),\Box_{0})=0.
\]
Since $(E,H)$ is strongly strictly positive in the sense of Nakano, 
\[
(S^{l}E,H^{l}_{\varepsilon})
\] 
is strictly positive in the sense of Nakano on $Y^{\prime}$ for every $\varepsilon$ by Proposition \ref{p11}, (e). Indeed, in the proof of (e) it is easy to see that if the positivity of $(E,H)$ is strict, the regularising sequence $\{H^{l}_{\varepsilon}\}$ will satisfies that

(e') for every relatively compact subset $Y^{\prime\prime}\subset\subset Y^{\prime}$ and every $l$,
\[
  i\Theta_{S^{l}E,H^{l}_{\varepsilon}}\geqslant C_{l}(1-\varepsilon)\omega\otimes\mathrm{Id}_{S^{l}E}
  \]
  over $Y^{\prime\prime}$ in the sense of Nakano. 

Now we apply Bochner's formula (\ref{e24}) to $(S^{l}E,H^{l}_{\varepsilon})$ on $Y^{\prime}$. Notice that $Y-Y^{\prime}$ is a closed subvariety hence has real codimension $\geqslant2$. In particular, the integral equality in (\ref{e24}) holds here. For any 
\[
\alpha\in\mathcal{H}^{n,q}(Y,S^{l}E(S^{l}H),\Box_{0}),
\] 
we have
\[
0=\lim(\|\partial^{\ast}_{H^{l}_{\varepsilon}}\alpha\|^{2}_{H^{l}_{\varepsilon}}+<ie(\Theta_{S^{l}E,H^{l}_{\varepsilon}})\Lambda\alpha,\alpha>_{H^{l}_{\varepsilon}}).
\]
Since $(S^{l}E,H^{l}_{\varepsilon})$ is strictly positive in the sense of Nakano, the Hermitian form $<ie(\Theta_{S^{l}E,H^{l}_{\varepsilon}})\Lambda\cdot,\cdot>_{H^{l}_{\varepsilon}}$ is positive-definite. Thus, we must have $\alpha=0$. The proof is complete.

\textbf{III}. Since $(E,H)$ is strictly positive in the sense of Griffiths, 
\[
(E\otimes\det E,H\otimes\det H)
\] 
is strongly strictly positive in the sense of Nakano due to \cite{Wu20b}, Theorem 1.3. Then we apply II on $(E\otimes\det E,H\otimes\det H)$ to obtain the desired conclusion.
\end{proof}

\address{

\small Current address: School of Mathematical Sciences, Fudan University, Shanghai 200433, People's Republic of China.

\small E-mail address: jingcaowu13@fudan.edu.cn
}

\end{document}